\documentclass[a4paper,10pt]{amsart}
\usepackage[utf8]{inputenc}
\usepackage{amsxtra}
\usepackage{amsopn}
\usepackage{amsmath,amsthm,amssymb}
\usepackage{amscd}
\usepackage{amsfonts}
\usepackage{caption}
\usepackage{latexsym}
\usepackage{verbatim}
\usepackage{MnSymbol}
\usepackage{xcolor}
\usepackage{tikz}
\usepackage{tikz-cd}
\usepackage{rotating}
\usepackage{hyperref}
\usepackage[shortlabels]{enumitem}
\usepackage{booktabs}

\theoremstyle{plain}
\newtheorem{theorem}{Theorem}[section]

\newtheorem{proposition}[theorem]{Proposition}
\newtheorem{corollary}[theorem]{Corollary}
\theoremstyle{definition}
\newtheorem{definition}[theorem]{Definition}
\newtheorem{remark}[theorem]{Remark}

\newtheorem{problem}[theorem]{Problem}

\newcommand\C{{\mathbb C}}

\renewcommand\phi{{\varphi}}
\newcommand\B{{\mathcal B}}

\newcommand\Z{{\mathbb Z}}

\renewcommand\H{{\mathcal H}}

\newcommand{\del}{\partial}
\newcommand{\delbar}{{\overline{\del}}}

\DeclareMathOperator{\Lie}{Lie}

\DeclareMathOperator{\rk}{rk}

\DeclareMathOperator{\im}{im\,}

\let\phi\varphi
\let\c\overline
\title[On the cohomology of the Bigolin complex]{On the cohomology of the Bigolin complex}

\author{Riccardo Piovani}
\address{Dipartimento di Matematica \lq\lq G. Peano",
Universit\`{a} degli Studi di Torino, Via Carlo Alberto 10,
10123 Torino, Italy}
\email{riccardo.piovani@unito.it}

\thanks{\newline 
The author is partially supported by GNSAGA of INdAM% and by University of Parma through the action Bando di Ateneo 2023 per la ricerca
}

\begin{document}

\begin{abstract}
Given a compact complex manifold, we study the cohomology and the Hodge theory for the elliptic complex of differential forms defined by Bigolin in 1969 and recently referred to as the Schweitzer complex.
Recall that the double complex of a compact complex manifold decomposes into a direct sum of so-called squares and zigzags, and the zigzags are the only components contributing to cohomology. The main result of this paper states that in complex dimension 3, the multiplicities of zigzags in this decomposition are completely characterised by Betti, Hodge, Aeppli numbers plus Bigolin numbers, namely the dimensions of the Bigolin cohomology. The result is sharp, meaning that if we remove Hodge or Bigolin numbers from the previous statement then it becomes false. In addition, we compute the Bigolin cohomology on the small deformations of the complex structure of the Iwasawa manifold, and then apply the main theorem to fully describe the double complexes of all the small deformations. We also prove a Hodge decomposition for Bigolin harmonic forms on compact K\"ahler manifolds of any dimension. Finally, we partially extend the definition of this complex on almost complex manifolds, providing a cohomological invariant on $1$-forms which is finite when the manifold is compact.
\end{abstract}

\maketitle

%\tableofcontents
%\addtocontents{toc}{\protect\setcounter{tocdepth}{1}}

\section{Introduction}
Given a complex manifold $M$, the exterior derivative on $(p,q)$-forms splits as
\[
d=\del+\delbar:A^{p,q}\to A^{p+1,q}\oplus A^{p,q+1},
\]
so that $d^2=0$ is equivalent to the relations
\[
\del^2=\delbar^2=\del\delbar+\delbar\del=0.
\]
Therefore, there are natural complexes of differential forms arising on $M$. The simplest is the \emph{Dolbeault complex}
\[
\dots\longrightarrow A^{p,q-1}\overset{\delbar}{\longrightarrow} A^{p,q}\overset{\delbar}{\longrightarrow} A^{p,q+1}{\longrightarrow}\dots
\]
together with its conjugate 
\[
\dots\longrightarrow A^{p-1,q}\overset{\del}{\longrightarrow} A^{p,q}\overset{\del}{\longrightarrow} A^{p+1,q}{\longrightarrow}\dots
\]
 The associated cohomology spaces are the \emph{Dolbeault cohomology} and the \emph{$\del$-cohomology}
\begin{align*}
H^{p,q}_\delbar:=\frac{\ker\delbar\cap A^{p,q}}{\delbar A^{p,q-1}}, &&
H^{p,q}_\del:=\frac{\ker\del\cap A^{p,q}}{\del A^{p-1,q}}.
\end{align*}
There is another natural complex, which will be the main subject of this paper
\begin{equation}\label{eq_big_complex}
\begin{tikzcd}
\dots\arrow[d,"\del \oplus d\oplus d \oplus \delbar"]\\
 A^{p-2,q}\oplus A^{p-1,q-1}\oplus A^{p,q-2}\arrow[d,"\del \oplus d \oplus \delbar"]\\
 A^{p-1,q}\oplus A^{p,q-1}\arrow[d,"\del \oplus \delbar"]\\
 A^{p,q}\arrow[d,"\del\delbar"]\\
 A^{p+1,q+1}\arrow[d,"d"]\\
 A^{p+1,q+2}\oplus A^{p+2,q+1}\arrow[d,"d"]\\
 A^{p+1,q+3}\oplus A^{p+2,q+2}\oplus A^{p+3,q+1}\arrow[d,"d"]\\
\dots
\end{tikzcd}
\end{equation}
\begin{comment}
\[
\dots\longrightarrow A^{p-2,q}\oplus A^{p-1,q-1}\oplus A^{p,q-2}\overset{\del\oplus d\oplus\delbar}{\longrightarrow} A^{p-1,q}\oplus A^{p,q-1}\overset{\del\oplus\delbar}{\longrightarrow} A^{p,q}\overset{\del\delbar}{\longrightarrow} A^{p+1,q+1}
\]
\[
 A^{p,q}\overset{\del\delbar}{\longrightarrow} A^{p+1,q+1}\overset{\del+\delbar}{\longrightarrow} A^{p+1,q+2}\oplus A^{p+2,q+1}\overset{\del+\delbar}{\longrightarrow} A^{p+1,q+3}\oplus A^{p+2,q+2}\oplus A^{p+3,q+1}
 {\longrightarrow}\dots
\]
\end{comment}
and which, in these notes, will be referred to as the \emph{Bigolin complex}, for reasons which hopefully will become clear soon. Its complete definition is provided in Section \ref{section_complex}.

We provide a short historical review of the study of the Bigolin complex.
It was introduced by Bigolin in \cite{Bi,Bi2} as a tool to study the Aeppli and Bott-Chern cohomology groups
\begin{align*}
H^{p,q}_A:=\frac{\ker \del\delbar\cap A^{p,q}}{\del A^{p-1,q}+\delbar A^{p,q-1}},&&
H^{p+1,q+1}_{BC}:=\frac{\ker d\cap A^{p+1,q+1}}{\del\delbar A^{p,q}},
\end{align*}
\textit{i.e.}, the cohomology of the  central part of the complex \eqref{eq_big_complex}, namely the shorter complex given by the map $\del\delbar$ and the two adjacent ones. Notice that in \cite[proof of Teorema 2.1]{Bi} it is proved that the cohomology of the complex \eqref{eq_big_complex} is finite dimensional on compact complex manifolds. Later, the hypercohomology of the complex \eqref{eq_big_complex} was studied by Demailly\footnote{The bibliography of \cite[Section 4]{S} refers to [J.P. Demailly, Analytic Geometry, 1993].} in his book \cite[Chapter VI, Section 12]{De} and by Schweitzer in his thesis \cite[Section 4]{S}. In \cite[Chapter VI, proof of Theorem 12.4]{De} there is a second proof of the finite dimensionality of the cohomology of the complex \eqref{eq_big_complex} on compact complex manifolds.

Recently, the complex \eqref{eq_big_complex} has been the central subject of the paper \cite{Ste} by Stelzig, where it is referred to as the \emph{Schweitzer complex}. Stelzig shows that the Schweitzer complex is elliptic \cite[Lemma 2.4]{Ste} (this provides a third proof of the finite dimensionality of the cohomology of the complex \eqref{eq_big_complex} on compact complex manifolds), proves an index theorem \cite[Theorem 2.8]{Ste} and other fundamental properties like Poincaré duality \cite[Corollary C]{Ste}. In \cite{MS}, a generalised version of Massey products has been defined starting from this Schweitzer complex. The cohomology spaces arising from the Schweitzer complex appear also in \cite{Ste23}. This complex is then considered in \cite{YY} to study bimeromorphic invariants associated to Bott-Chern hypercohomology. A related $L^2$ Hilbert complex has been defined in \cite{HP}, where \eqref{eq_big_complex} is called \emph{Aeppli-Bott-Chern complex}, since the focus of the paper is the  central part of the complex and its $L^2$ cohomology. %In the same paper the authors wish for the establishment of an $L^2$ index theorem for the complex \eqref{eq_big_complex} on Galois coverings of a compact complex manifold, analogous to the Atiyah $L^2$ index theorem \cite{Ati}.

Apart from the above mentioned works, the Bigolin complex has received very little attention, as noticed also in \cite{Ste}. For instance, in the literature there are no explicit examples of calculation of Bigolin numbers, namely the dimensions of the Bigolin cohomology spaces, \textit{i.e.}, of the cohomology of the Bigolin complex. The only well studied part of the complex is the central part, related to Aeppli and Bott-Chern cohomology. In view of these considerations, the main purposes of this paper are the following: to settle the fundamental Hodge theoretic properties of the extremal parts of the Bigolin complex; to prove that these Bigolin numbers are useful to describe the double complex of differential $(p,q)$-forms associated to a compact complex manifold of complex dimension 3; to compute the cohomological invariants of the Bigolin complex on examples of compact complex manifolds;  finally, to suggest the study of a new cohomological invariant related to 1-forms which can be defined on compact almost complex manifolds. %Meanwhile, we present several graphic representations of the Bigolin complex, hoping for an easier understanding of its nature.

The paper is organised as follows.
In Section \ref{section_complex}, for every choice of a bidegree $(p,q)$ on a complex manifold $M$, we define the Bigolin complex, denoted by $(\B^\bullet_{p,q},\delta^\bullet_{p,q})$, and its cohomology
\[
H^\bullet_{p,q}:=\frac{\ker\delta^\bullet_{p,q}\cap \B^\bullet_{p,q}}{\delta^{\bullet-1}_{p,q} \B^{\bullet-1}_{p,q}}.
\]
We then present several relations between this Bigolin cohomology and other cohomology spaces, which follow directly from their definitions.
%There is no original material in this section other than, possibly, the observations contained in Remarks \ref{rmk map between cohomologies} and \ref{rmk algebra structure}.

Section \ref{sec_hodge} is first devoted to the definition of the elliptic operators 
\[
\square^k_{p,q}:\B^k_{p,q}\to\B^k_{p,q}
\]
which are naturally associated to the elliptic complex (cf. \cite[Proof of Corollary D]{Ste}, \cite[Section 2]{HP}, \cite[Section 1]{V}). Given a Hermitian metric $g$ on $M$, the operator $\square^k_{p,q}$ is defined via the differential 
\[
\delta^k_{p,q}:\B^k_{p,q}\to\B^{k+1}_{p,q}
\] 
and its formal adjoint 
\[
(\delta^k_{p,q})^*:\B^{k+1}_{p,q}\to\B^{k}_{p,q}.
\]
If the manifold $M$ is compact, then 
 \[
 \ker\square^k_{p,q}=\ker\delta^k_{p,q}\cap\ker(\delta^{k-1}_{p,q})^* \simeq H^k_{p,q}
 \]
  and it has finite dimension $h^k_{p,q}$. Furthermore, Hodge theory provides several $L^2$ orthogonal decompositions of the spaces of forms $\B^k_{p,q}$.% which, to the knowledge of the author, have not been reported elsewhere.
\begin{theorem}[{see Theorem \ref{theorem hodge decomposition}}]\label{theorem_intro_decomp}
Given a compact Hermitian manifold $(M,g)$, there are $L^2$ orthogonal decompositions
\[
\B^k_{p,q}=\ker\square^k_{p,q}\oplus\delta^{k-1}_{p,q}\B^{k-1}_{p,q}\oplus(\delta^{k}_{p,q})^*\B^{k+1}_{p,q}.
\]
\end{theorem}
We also show that the Hodge * operator and conjugation induce a symmetry for the invariants $h^k_{p,q}$ and consequently for the Euler characteristic $\chi_{p,q}$ of the complex, namely
\[
h^k_{p,q}=h^{2n-k-1}_{n-q-1,n-p-1}=h^{2n-k-1}_{n-p-1,n-q-1}=h^k_{q,p},
\]
\[
\chi_{p,q}=-\chi_{n-q-1,n-p-1}=-\chi_{n-p-1,n-q-1}=\chi_{q,p}.
\]
This symmetry is also proved in \cite[Corollary C]{Ste} by another method\footnote{The difference in the indices is compatible with the different definitions of the complex in the two papers.}.

We recall that on a compact K\"ahler manifold we have the Hodge decomposition for de Rham harmonic forms: $\ker\Delta_d\cap A^k_\C=\oplus_{p+q=k}\ker \Delta_\delbar\cap A^{p,q}$, where $\Delta_D=DD^*+D^*D$ and $D^*$ denotes the formal adjoint of $D$ for $D\in\{d,\delbar,\del\}$.
By applying K\"ahler identities, we are able to prove the following Hodge decomposition for the harmonic forms of the complex $(\B^\bullet_{p,q},\delta^\bullet_{p,q})$.
\begin{theorem}[{see Theorem \ref{thm_kahler}}]\label{thm intro kahler}
Given a compact K\"ahler manifold $(M,g)$, it holds
\[
\ker \square^k_{p,q}=\ker\Delta_\delbar\cap\B^k_{p,q}.
\]
\end{theorem}
The proof follows by a formula relating the Laplacians $\square^k_{p,q}$ and $\Delta_\delbar$ under the K\"ahler assumption, which was inspired by \cite[proof of Theorem 8.8]{HP}.
%The cohomological implications of Theorem \ref{thm intro kahler} are not a novelty, as it is explained in Remark \ref{rmk del delbar lemma}, while the theorem itself is original. 
Apart from Theorem \ref{thm intro kahler}, the remaining content of Section \ref{sec_hodge}, including Theorem \ref{theorem_intro_decomp}, was more or less already known (cf. \cite[proofs of Corollaries B, D]{Ste}). Nevertheless, this content is not available elsewhere in the literature and thus we think that its presentation in these notes may be convenient for further studies.

In Section \ref{section compact} we deduce further cohomological relations involving Bigolin cohomology under the additional assumption of the compactness of the complex manifold and we use these relations to count, in any dimension, the number of the invariants $h^k_{p,q}$ which are, a priori and modulo their symmetries (conjugation, Hodge $*$ duality), interesting to study and different from Betti, Dolbeault, Aeppli and Bott-Chern numbers. This count is also fundamental in the next two sections for complex dimension 3.

In Section \ref{section double complex} we prove the main result of the paper. It is well known that the double complex $(A^{\bullet,\bullet},\del,\delbar)$ of a compact complex manifold decomposes into a direct sum of so-called squares and finitely many zigzags. The zigzags are the only components contributing to cohomology. More precisely, in \cite[Proposition 6 and Lemma 8]{Ste21} it is shown how every zigzag contributes to de Rham, Dolbeault, $\del$-, Aeppli and Bott-Chern cohomology. The analogue result for the Bigolin cohomology is Proposition \ref{prop zigzag contribute to bigolin}.

We consider a compact complex manifold of complex dimension 3 and classify 17 different shapes of zigzags (18 if we include the dot counting the connected components of the manifold) which exhaust all the possible zigzags modulo conjugation and duality. We then describe how these zigzags contribute to the Bigolin cohomology $H^k_{p,q}$ and are actually able to reverse the relation \lq\lq zigzags $\implies$ cohomology" as follows.
\begin{theorem}[see Theorem \ref{thm zigzags 3manifolds}]\label{thm intro zigzags 3manifolds}
Given a compact complex 3-manifold, the multiplicities of the zigzags in the direct sum decomposition of the double complex $(A^{\bullet,\bullet},\del,\delbar)$ are completely characterised by the Hodge numbers $h^{p,q}_\delbar$, Aeppli numbers $h^{p,q}_A$, Betti numbers $b^k$ 
and Bigolin numbers $h^k_{p,q}$ for all $p,q,k$.
\end{theorem}
This result is sharp, meaning that if we remove either  the Bigolin numbers $h^k_{p,q}$ or the Hodge numbers $h^{p,q}_\delbar$ from the statement then it becomes false, as it is shown respectively in Remarks \ref{remark necessity Bigolin} and \ref{remark necessity of hodge numbers}. Note that Aeppli numbers $h^{p,q}_A$ and Betti numbers $b^k$ can be seen as a particular case of Bigolin numbers $h^k_{p,q}$ (cf. Section \ref{section_complex}). In particular this theorem answers, in the special case of complex dimension 3, the following question posed by Jonas Stelzig in a private communication: \lq\lq If we know the dimensions of all Bigolin cohomology groups, can we calculate the multiplicities of all zigzags?".
The techniques we used are new, and can similarly be extended in higher dimensions, although computational complexity rapidly grows. Preliminary computations seem to show that the same result of Theorem \ref{thm intro zigzags 3manifolds} does not directly generalise to higher dimension: it is likely that more cohomological invariants are needed to completely describe the multiplicities of zigzags on a general compact complex manifold, given the high rate of growth of the number of zigzags.
Note that in complex dimension 2 the double complexes have more restrictions, as explained in \cite[Remark 19]{Ste21}, therefore the situation is simpler.

In Section \ref{sec_iwasawa} we make use of recent results of \cite{Ste21,Ste,Ste23} in order to explicitly compute the Bigolin numbers $h^k_{p,q}$ of the small deformations of the Iwasawa manifold reducing the problem to a linear algebra calculation. The Dolbeault cohomology of these small deformations were computed in \cite{N}, while the Aeppli and Bott-Chern cohomology in \cite{A}. It has been proved that in this example Aeppli and Bott-Chern cohomology can distinguish between complex structures more precisely than Dolbeault cohomology, providing finer invariants. Here we show that, in this example, the Bigolin cohomology $H^{k}_{p,q}$ is  as powerful as Aeppli and Bott-Chern cohomology, in order to distinguish classes of complex structures. We then apply Theorem \ref{thm zigzags 3manifolds} to completely describe the double complexes of all the small deformations of the Iwasawa manifold.

Finally, Section \ref{sec_almost} is devoted to a partial generalisation of the elliptic complex \eqref{eq_big_complex} in the case of almost complex manifolds. We observe that there are elliptic complexes
\[
A^{0,0}\overset{d}{\longrightarrow}A^{0,1}\oplus A^{1,0}\overset{\del\oplus\delbar}{\longrightarrow} A^{1,1},
\]
\[
A^{n-1,n-1}\overset{d}{\longrightarrow}A^{n-1,n}\oplus A^{n,n-1}\overset{d}{\longrightarrow} A^{n,n},
\]
which provide new invariants of almost complex manifolds. The cohomology of these complexes are
\[
H^1_{\B}:=\frac{\ker \del\oplus\delbar\cap A^{0,1}\oplus A^{1,0}}{d A^{0,0}},
\]
\[
H^{2n-1}_{\B}:=\frac{\ker d\cap A^{n-1,n}\oplus A^{n,n-1}}{d A^{n-1,n-1}},
\]
and when the manifold is compact they have finite dimension $h^1_\B=h^{2n-1}_\B$, since they are isomorphic to the kernels of elliptic operators. To the knowledge of the author, the invariant $h^1_\B$ is the only finite invariant of the (non integrable) almost complex structure of a compact manifold which is \emph{cohomological} (it is the dimension of the cohomology space of an elliptic differential complex). We refer to \cite{PT4,ST} for two other finite invariants defined on $1$-forms which are only trivially cohomological \footnote{Trivially cohomological here means that one of the two maps of the complex defining the cohomology is zero.} \footnote{Another classical trivially cohomological invariant is $h^{p,0}_\delbar$, defined in the next paragraph.}.

This problem of finding finite invariants of compact almost complex manifolds which are cohomological is strictly related to the following problem of interest in the last few years.
Denote by $h^{p,q}_\delbar$ the dimension of the Dolbeault cohomology $H^{p,q}_\delbar\simeq\ker\Delta_\delbar\cap A^{p,q}$ of a compact complex manifold. Problem 20 by Kodaira and Spencer in \cite{Hi} basically asks: \lq\lq given a compact almost complex manifold $(M,J)$ endowed with an almost Hermitian metric $g$, is the number $h^{p,q}_\delbar:=\dim_\C\ker\Delta_\delbar\cap A^{p,q}$ metric independent (as in the case of complex manifolds)? In case it is not independent, give a generalisation of the $h^{p,q}_\delbar$ of a complex manifold depending only on $(M,J)$." The first part of the problem has been solved in \cite{HZ} and later in \cite{TT}, where it is shown that $h^{p,q}_\delbar$ depends on the metric. Concerning the second part of the problem, in \cite{CW} a Dolbeault cohomology for almost complex manifolds has been defined, but in \cite{CPS} it is shown that it can be infinite dimensional even when $M$ is compact. Different Aeppli and Bott-Chern-like finite dimensional (when $M$ is compact) invariants have been defined as the dimension of the kernel of elliptic operators in \cite{PT4,ST}, but in general they depend on the metric.

Given a compact manifold, Hodge theory provides an isomorphism between the kernel of an elliptic operator, defined via the combination of a metric independent elliptic complex and a metric, and the cohomology of the complex itself. In this way, a metric independent elliptic complex gives rise to metric independent and finite invariants of the compact manifold. However, the problem on almost complex manifolds is that the conditions $\del^2\ne0$, $\delbar^2\ne0$, $\del\delbar+\delbar\del\ne0$ drastically reduce the number of elliptic complexes. For this reason $h^1_\B$, being defined via a metric independent elliptic complex, is a notable exception.

\medskip
\noindent{\em Acknowledgements.}
I thank Adriano Tomassini and Lorenzo Sillari for useful discussions. Furthermore, I am sincerely grateful to Jonas Stelzig for many interesting remarks on the topic and for explaining me some results in \cite{Ste22} which have been used in Remarks \ref{remark necessity Bigolin} and \ref{remark necessity of hodge numbers}.

\section{The complex}\label{section_complex}

\begin{definition}[The Bigolin complex]
Given a complex manifold $M^n$ of complex dimension $n$, for each bidegree of integers $(p,q)$ with $-1\le p,q\le n$, we define a complex $(\B^\bullet_{p,q},\delta^\bullet_{p,q})$ of differential forms. Denoting by $A^{r,s}$ the space of $(r,s)$-forms, and by $d$ the exterior derivative, for every $k\in\Z$ we set
\begin{align*}
&\B_{p,q}^k:=\bigoplus_{\substack{r+s=k\\ 0\le r\le p\\0\le s\le q}} A^{r,s}& k\le p+q,\\
&\B_{p,q}^k:=\bigoplus_{\substack{r+s=k+1\\ p< r\le n\\ q< s\le n}} A^{r,s}& k> p+q,
\end{align*}
with differentials $\delta^k_{p,q}:\B_{p,q}^k\to\B_{p,q}^{k+1}$ defined by
\begin{align*}
&\delta^k_{p,q}=\pi_{\B_{p,q}^{k+1}}\circ d& k< p+q,\\
&\delta^k_{p,q}=\del\delbar& k= p+q,\\
&\delta^k_{p,q}=d& k> p+q,
\end{align*}
where $\pi_{\B_{p,q}^{k}}$ is the natural projection from $A^\bullet_\C:=A^\bullet\otimes\C$ to $\B_{p,q}^{k}$. The complex $(\B^\bullet_{p,q},\delta^\bullet_{p,q})$ will be referred to as the \emph{Bigolin complex}.

Note that our definition of $(\B^\bullet_{p,q},\delta^\bullet_{p,q})$ is shifted in the indexes $(p,q)$ with respect to the definitions of \cite{S,De,Ste}: denoting by $(\mathcal{L}^\bullet_{p,q},d^\bullet_{p,q})$ the complex defined in these papers, it holds $(\mathcal{L}^\bullet_{p+1,q+1},d^\bullet_{p+1,q+1})=(\B^\bullet_{p,q},\delta^\bullet_{p,q})$. This shift in the index is motivated by the fact that, thanks to duality (cf. Remark \ref{remark duality}), we are interested in studying only the first part of the complex $(\B^\bullet_{p,q},\delta^\bullet_{p,q})$, for $k\le p+q$.
\end{definition}

\begin{remark}[The range of $k$]
Depending on the values of $(p,q)$, there are slightly different ranges for $k$ for which the space $\B^k_{p,q}$ is non-zero:
\begin{itemize}
\item if $0\le p,q\le n-1$, then $\B^k_{p,q}\ne\{0\}$ for  $0\le k\le 2n-1$;
\item if $p=q=-1$, then $\B^k_{p,q}\ne\{0\}$ for $-1\le k\le 2n-1$;
\item if $p=q=n$, then $\B^k_{p,q}\ne\{0\}$ for $0\le k\le 2n$;
\item if $\min(p,q)=-1<\max(p,q)$, then $\B^k_{p,q}\ne\{0\}$ for $\max(p,q)\le k\le 2n-1$;
\item if $\max(p,q)=n>\min(p,q)$, then $\B^k_{p,q}\ne\{0\}$ for $0\le k\le p+q$.
\end{itemize}
Note that in the extremal cases $p=q\in\{-1,n\}$ the complex $(\B^\bullet_{p,q},\delta^\bullet_{p,q})$ just defined actually coincides with the de Rham complex $(A^\bullet,d)$. For this reason, in the rest of the paper we will generally focus on the case $0\le k\le 2n-1$ and $p+q\notin\{-2,2n\}$.
\end{remark}

\begin{remark}[Graphic representation of the complex]\label{remark_graphic_rep}
We provide some graphic explanation of how this complex is defined.
If $n=10$, $p=4$ and $q=7$, we visualise the complex as follows:
\begin{center}
\begin{tikzpicture}[scale=0.5]
    \foreach \x in {0,1,...,10}{
    \draw (\x,0)node[below,font=\footnotesize] {\x};
    }
    \foreach \y in {0,1,...,10}{
    \draw (0,\y)node[left,font=\footnotesize] {\y};
    }
    \foreach \x in {0,1,...,10}{
    \foreach \y in {0,1,...,10}{
    \fill[black!30] (\x,\y) circle[radius=1pt];
    }}
    \foreach \x in {0,1,...,4}{
    \foreach \y in {0,1,...,7}{
    \fill[black!80] (\x,\y) circle[radius=2pt];
    }}
    \foreach \x in {5,6,...,10}{
    \foreach \y in {8,9,...,10}{
    \fill[black!80] (\x,\y) circle[radius=2pt];
    }}
\end{tikzpicture}\\
\begin{tikzpicture}[scale=1]
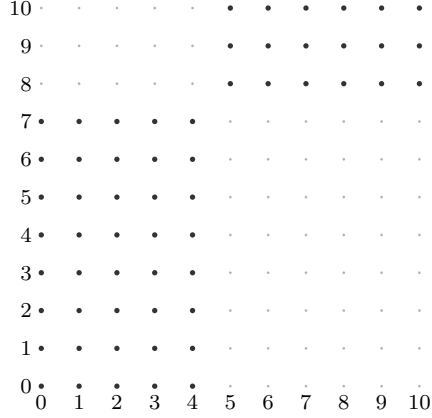
    \draw (0,0)node[below,font=\footnotesize] {Figure 1: the complex $(\B^\bullet_{4,7},\delta^\bullet_{4,7})$ for $n=10$.};
\end{tikzpicture}
\end{center}
Each dot in the diagram corresponds to a space $A^{r,s}$, where $r$ varies in the horizontal axis and $s$ in the vertical axis. The operators $\del$ and $\delbar$ map a point with coordinates $(r,s)$ to points with coordinates respectively $(r+1,s)$ and $(r,s+1)$.
In the picture we can see two rectangles made of thick points. For $0\le k\le 11$, the space $\B_{4,7}^k$ corresponds to all the thick points (in the bottom-left rectangle) with coordinates $(r,s)$ such that $r+s=k$. On the other hand, for $12\le k\le 19$, the space $\B_{7,4}^k$ corresponds to all the thick points (in the top-right rectangle) with coordinates $(r,s)$ such that $r+s=k+1$.
For instance, we see that $\B_{4,7}^9=A^{2,7}\oplus A^{3,6}\oplus A^{4,5}$ and $\B_{7,4}^{13}=A^{5,9}\oplus A^{6,8}$.
\end{remark}

\begin{remark}[Explicit formulation of $\delta^k_{p,q}$ for $k<p+q$]\label{remark_differential}
It is immediate from the graphic description that for $0\le k<\min(p,q)$ the differential $\delta^k_{p,q}$ equals the exterior derivative $d$. %For the remaining case $\min(p,q)\le k<p+q$, we assume for simplicity that $p<q$. 
If $\min(p,q)=p\le k< q=\max(p,q)$, then the differential $\delta^k_{p,q}$ acts on
\[
\B^{k}_{p,q}=A^{0,k}\oplus\dots\oplus A^{k-p-1,p+1}\oplus A^{k-p,p}
\]
component by component as
\[
\delta^k_{p,q}= d\oplus\dots\oplus d\oplus\delbar.
\]
On the other hand, if $\min(p,q)=q\le k< p=\max(p,q)$, then the differential $\delta^k_{p,q}$ acts on
\[
\B^{k}_{p,q}=A^{k-q,q}\oplus A^{k-q-1,q+1}\oplus\dots\oplus A^{k-p,p}
\]
component by component as
\[
\delta^k_{p,q}= \del\oplus d\oplus\dots\oplus d.
\]
Finally, if $\max(p,q)\le k< p+q$, then the differential $\delta^k_{p,q}$ acts on
\[
\B^{k}_{p,q}=A^{k-q,q}\oplus A^{k-q+1,q-1}\oplus\dots\oplus A^{p-1,k-p+1}\oplus A^{p,k-p}
\]
component by component as
\[
\delta^k_{p,q}=\del\oplus d\oplus\dots\oplus d\oplus \delbar.
\]
\end{remark}

\begin{definition}[Bigolin cohomology]
It is easy to check that $(\B^\bullet_{p,q},\delta^\bullet_{p,q})$ is actually a complex, that is $\delta_{p,q}^{k}\circ\delta_{p,q}^{k-1}=0$. Therefore, we can define the cohomology spaces
\[
H^{k}_{p,q}:=\frac{\ker \delta^k_{p,q}}{\im \delta^{k-1}_{p,q}}.
\]
This will be referred to as the \emph{Bigolin cohomology}.
\end{definition}

We now examine some relations of Bigolin cohomology with itself and with other cohomology spaces which directly follow from their definitions.

\begin{remark}[Conjugation]
Observe that conjugation 
\begin{align*}
A^{r,s}&\to A^{s,r}\\
\alpha&\mapsto\c\alpha
\end{align*} 
induces an isomorphism on cohomology, namely for all $k$
\[
H^k_{p,q}\simeq H^k_{q,p}.
\]
\end{remark}

\begin{remark}[Aeppli and Bott-Chern cohomologies]%\label{remark_cohomology}
We see that for $k=p+q$ this space coincides with the Aeppli cohomology
\[
H^{p+q}_{p,q}=H^{p,q}_A:=\frac{\ker \del\delbar\cap A^{p,q}}{\del A^{p-1,q}+\delbar A^{p,q-1}},
\]
while for $k=p+q+1$ it equals the Bott-Chern cohomology
\[
H^{p+q+1}_{p,q}=H^{p+1,q+1}_{BC}:=\frac{\ker d\cap A^{p+1,q+1}}{\del \delbar A^{p,q}}.
\]
\end{remark}

\begin{proposition}[De Rham and Dolbeault cohomologies]\label{proposition cohomology}
Given a complex manifold $M^n$ of complex dimension $n$, then
\begin{enumerate}
\item if $k<\min(p,q)$, then $H^{k}_{p,q}= H^{k}_{dR}$;
\item if $k>n+\max(p,q)$ , then $H^{k}_{p,q}= H^{k+1}_{dR}$;
\item if $k=\min(p,q)$, then $H^{k}_{p,q}\supseteq H^{k}_{dR}$;
\item if $0\le k\le n$, then $H^k_{0,n}=H_\delbar^{0,k}$;
\item if $0\le q\le n-1$ and $0\le k\le q-1$, then $H^k_{0,q}=H_\delbar^{0,k}$;
\item if $n-1\le k\le 2n-1$, then $H^k_{n-1,-1}=H_\delbar^{n,k+1-n}$;
\item if $0\le q\le n-2$ and $n+q+1\le k\le 2n-1$, then $H^k_{n-1,q}=H_\delbar^{n,k+1-n}$.
\end{enumerate}
By conjugation, similar relations can be obtained replacing the Dolbeault cohomology with the $\del$ cohomology.
\end{proposition}
\begin{proof}
The first three points follow directly from the definitions, since
\begin{itemize}
\item $\B^k_{p,q}=A^k_\C$ for $k\le\min(p,q)$;
\item $\delta^k_{p,q}=d_{|A^k_\C}$ for $k<\min(p,q)$;
\item $\B^k_{p,q}=A^{k+1}_\C$ for $k\ge 2n-1-\min(n-p-1,n-q-1)=n+\max(p,q)$;
\item $\delta^k_{p,q}=d_{|A^{k+1}_\C}$ for $k\ge n+\max(p,q)$.
\end{itemize}
The last four points follow similarly, checking the definitions.
\end{proof}
It will follow from an explicit computation in Section \ref{sec_iwasawa} that in general the inclusion \emph{(3)} is not an equality. 

There are some other relations of the Bigolin cohomology with itself; by conjugation, it is enough to consider the case $p\le q$.

\begin{proposition}[Relations of Bigolin cohomology]\label{proposition bigolin cohomology}
Given a complex manifold $M^n$ of complex dimension $n$, then
\begin{enumerate}
\item if $p\le k\le q-2$, then $H^k_{p,q}=H^k_{p,q-1}$;
\item if $n+p+2\le k\le 2n-1$, then $H^k_{p,q}=H^k_{p+1,q}$.
\end{enumerate}
\end{proposition}
\begin{proof}
It follows from the observation that
\begin{itemize}
\item $\B^k_{p,q}=\B^k_{p,q-1}$ for $0\le k\le q-1$;
\item $\delta^k_{p,q}=\delta^k_{p,q-1}$ for $0\le k\le q-2$;
\item $\B^k_{p,q}=\B^k_{p+1,q}$ for $n+p+1\le k\le 2n-1$;
\item $\delta^k_{p,q}=\delta^k_{p+1,q}$ for $n+p+1\le k\le 2n-1$.\qedhere
\end{itemize}
\end{proof}

Finally, we observe that it is possible to partially generalise a map between cohomologies and an algebra structure from the Aeppli-Bott-Chern case to the Bigolin case.

\begin{remark}[Map between cohomologies]\label{rmk map between cohomologies}
We recall that there is always a map between Bott-Chern and Aeppli cohomologies induced by the identity
\[
H^{p,q}_{BC}\to H^{p,q}_A.
\]
We can generalise this map as follows. If $0\le p,q\le n$, $p+q< k\le 2n-2$ and $k-q,k-p\le n$, then
\[
\B^k_{p,q}=\B^{k+1}_{k-q,k-p}.
\]
We may see the equality $\B^5_{1,2}=\B^6_{3,4}=A^{2,4}\oplus A^{3,3}$ for $n=5$ as follows.
\begin{align*}
\begin{tikzpicture}[scale=0.5]
    \foreach \x in {0,1,...,5}{
    \draw (\x,0)node[below,font=\footnotesize] {\x};
    }
    \foreach \y in {0,1,...,5}{
    \draw (0,\y)node[left,font=\footnotesize] {\y};
    }
    \foreach \x in {0,1,...,5}{
    \foreach \y in {0,1,...,5}{
    \fill[black!30] (\x,\y) circle[radius=1pt];
    }}
    \foreach \x in {0,1,...,1}{
    \foreach \y in {0,1,...,2}{
    \fill[black!80] (\x,\y) circle[radius=2pt];
    }}
    \foreach \x in {2,3,...,5}{
    \foreach \y in {3,4,...,5}{
    \fill[black!80] (\x,\y) circle[radius=2pt];
    }}
\end{tikzpicture}
& &
\begin{tikzpicture}[scale=0.5]
    \foreach \x in {0,1,...,5}{
    \draw (\x,0)node[below,font=\footnotesize] {\x};
    }
    \foreach \y in {0,1,...,5}{
    \draw (0,\y)node[left,font=\footnotesize] {\y};
    }
    \foreach \x in {0,1,...,5}{
    \foreach \y in {0,1,...,5}{
    \fill[black!30] (\x,\y) circle[radius=1pt];
    }}
    \foreach \x in {0,1,...,3}{
    \foreach \y in {0,1,...,4}{
    \fill[black!80] (\x,\y) circle[radius=2pt];
    }}
    \foreach \x in {4,5,...,5}{
    \foreach \y in {5}{
    \fill[black!80] (\x,\y) circle[radius=2pt];
    }}
\end{tikzpicture}
\end{align*}
It is then straightforward to verify that the following map, induced by the identity in cohomology, is well defined
\[
H^{k}_{p,q}\to H^{k+1}_{k-q,k-p}.
\]
\end{remark}

\begin{remark}[Algebra structure]\label{rmk algebra structure}
It is well known that the wedge multiplication induces a map
\[
H^{p,q}_{BC}\times H^{r,s}_A\to H^{p+r,q+s}_A.
\]
For the cohomology $H^k_{p,q}$, we observe that the wedge multiplication induces the following map:
\[
H^{k}_{p,q}\times H^{l}_{r,s}\to H^{k+l+1}_{p+r,q+s},
\]
for $k>p+q$ and $l>r+s$.
\end{remark}

\section{Hodge Theory}\label{sec_hodge}
One can see that $(\B^\bullet_{p,q},\delta^\bullet_{p,q})$ is an elliptic complex, as it is proved in \cite[Lemma 2.4]{Ste} (cf. \cite[Section 1]{V}). Therefore, once we endow our complex manifold $M^n$ with a Hermitian metric $g$, we can define elliptic operators which are naturally associated to the complex, defined as follows (cf. \cite[Section 1]{V}, \cite[proof of Corollary D]{Ste}, \cite[Sections 2, 10]{HP}).

\begin{remark}[Hodge * duality]
We recall that the Hodge * operator is a $\C$-linear isomorphism between $A^{r,s}$ and $A^{n-s,n-r}$. Then * induces an isomorphism between the spaces
\[
*:\B^{k}_{p,q}\simeq \B^{2n-k-1}_{n-q-1,n-p-1}.
\]
For the values $n=10$, $p=4$ and $q=7$, the following diagram represents the complex $(\B^\bullet_{n-q-1,n-p-1},\delta^\bullet_{n-q-1,n-p-1})=(\B^\bullet_{2,5},\delta^\bullet_{2,5})$, where the isomorphism becomes evident: it is a reflection of Figure 1 with respect to the diagonal $r+s=n$, where $r$ is the horizontal coordinate and $s$ is the vertical coordinate.
\medskip
\begin{center}
\begin{tikzpicture}[scale=0.5]
    \foreach \x in {0,1,...,10}{
    \draw (\x,0)node[below,font=\footnotesize] {\x};
    }
    \foreach \y in {0,1,...,10}{
    \draw (0,\y)node[left,font=\footnotesize] {\y};
    }
    \foreach \x in {0,1,...,10}{
    \foreach \y in {0,1,...,10}{
    \fill[black!30] (\x,\y) circle[radius=1pt];
    }}
    \foreach \x in {0,1,...,2}{
    \foreach \y in {0,1,...,5}{
    \fill[black!80] (\x,\y) circle[radius=2pt];
    }}
    \foreach \x in {3,4,...,10}{
    \foreach \y in {6,7,...,10}{
    \fill[black!80] (\x,\y) circle[radius=2pt];
    }}
\end{tikzpicture}\\
\begin{tikzpicture}[scale=1]
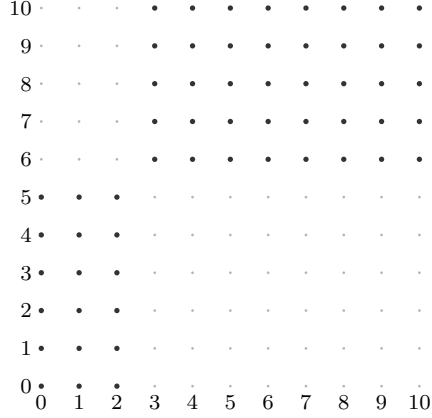
    \draw (0,0)node[below,font=\footnotesize] {Figure 2: the complex $(\B^\bullet_{2,5},\delta^\bullet_{2,5})$ for $n=10$.};
\end{tikzpicture}
\end{center}
\end{remark}

\begin{definition}[Codifferential]
Denote by
\[
(\delta^k_{p,q})^*:\B^{k+1}_{p,q}\to\B^k_{p,q}
\]
the formal adjoint of $\delta^k_{p,q}$, which we may also call codifferential. By applying the usual argument involving the Stokes theorem, it is possible to show
\[
(\delta_{p,q}^k)^*=-*\delta_{n-q-1,n-p-1}^{2n-k-1}*.
\]
In Remark \ref{remark_differential} we have a description of $\delta^k_{p,q}$ for $0\le k<p+q$. For $p+q<k\le 2n-1$, a similar description holds for the codifferential $(\delta^k_{p,q})^*$: in fact
\[
(\delta^k_{p,q})^*=\pi_{\B^k_{p,q}}\circ d^*.
\]
On the other hand, the codifferential for $0\le k<p+q$ is simpler, since
\[
(\delta^k_{p,q})^*=d^*.
\]
\end{definition}

\begin{definition}[Elliptic operators associated to the complex] We now define elliptic and formally self adjoint differential operators 
\[
\square^k_{p,q}:\B^k_{p,q}\to\B^k_{p,q}.
\]
We set, for every integer $0\le k\le 2n-1$ and $k\ne p+q,p+q+1$,
\[
\square^k_{p,q}:=(\delta^k_{p,q})^*\delta^k_{p,q}+\delta^{k-1}_{p,q}(\delta^{k-1}_{p,q})^*,
\]
which is a differential operator of order 2, while, for $k=p+q$,
\[
\square^k_{p,q}:=(\delta^k_{p,q})^*\delta^k_{p,q}+(\delta^{k-1}_{p,q}(\delta^{k-1}_{p,q})^*)^2=:\square_A,
\]
and, for $k=p+q+1$,
\[
\square^k_{p,q}:=((\delta^k_{p,q})^*\delta^k_{p,q})^2+\delta^{k-1}_{p,q}(\delta^{k-1}_{p,q})^*=:\square_{BC},
\]
which are of order 4. We refer to \cite{HP} for the terminology $\square_A$ and $\square_{BC}$, which is related to Aeppli and Bott-Chern cohomology/harmonic forms.
\end{definition}

When the Hermitian manifold $(M,g)$ is compact, by the theory of elliptic operators/complexes (cf. \cite{AB}, \cite{V}) we obtain the following fundamental result.

\begin{theorem}[Hodge decomposition]\label{theorem hodge decomposition}
Given a compact Hermitian manifold $(M,g)$, then the spaces of harmonic forms
\[
\H^k_{p,q}:=\ker\square^k_{p,q}=\ker\delta^k_{p,q}\cap\ker(\delta^{k-1}_{p,q})^*
\]
have finite dimension $h^k_{p,q}$ and they provide a $L^2$ orthogonal decomposition
\[
\B^k_{p,q}=\H^k_{p,q}\oplus\delta^{k-1}_{p,q}\B^{k-1}_{p,q}\oplus(\delta^{k}_{p,q})^*\B^{k+1}_{p,q}.
\]
\end{theorem}

\begin{corollary}[Hodge isomorphism]\label{corollary hodge isomorphism}
Given a compact Hermitian manifold $(M,g)$, then
\[
\ker \delta^k_{p,q}=\H^k_{p,q}\oplus\delta^{k-1}_{p,q}\B^{k-1}_{p,q},
\]
which implies
\[
H^k_{p,q}\cong \H^k_{p,q}.
\]
\end{corollary}
In particular, the numbers $h^k_{p,q}$ do not depend on the choice of the Hermitian metric $g$ which has been chosen to define them. They will be referred to as \emph{Bigolin numbers}.

\begin{remark}[Hodge * duality of harmonic forms]	\label{remark duality}
It is well known that the Hodge * operator induces an isomorphism between the spaces of Aeppli and Bott-Chern harmonic forms
\[
*:\H^{p+q}_{p,q}=\ker\square_A\cap A^{p,q}\simeq \ker\square_{BC}\cap A^{n-q,n-p}=\H_{n-q-1,n-p-1}^{2n-p-q-1},
\]
due to the formula
\begin{align*}
*\square_A=\square_{BC}*, && \square_A*=*\square_{BC}.
\end{align*}
For $k\ne p+q,p+q+1$, a straightforward computation shows
\begin{align*}
*\square_{p,q}^k=\square_{n-q-1,n-p-1}^{2n-k-1}*, && \square_{p,q}^k*=*\square_{n-q-1,n-p-1}^{2n-k-1},
\end{align*}
implying
\begin{equation}\label{eq_star_isom}
*:\H^{k}_{p,q}\simeq \H^{2n-k-1}_{n-q-1,n-p-1},
\end{equation}
which generalises the isomorphism between Aeppli and Bott-Chern harmonic forms.
\end{remark}

 Taking into account also the isomorphism given by the conjugation, we have the following relations for the Bigolin numbers $h^k_{p,q}$.

\begin{corollary}\label{corollary hodge numbers}
Given a compact complex manifold $M^n$ of complex dimension $n$, then
\begin{equation*}%\label{eq_hodge_numbers}
h^k_{p,q}=h^{2n-k-1}_{n-q-1,n-p-1}=h^{2n-k-1}_{n-p-1,n-q-1}=h^k_{q,p}.
\end{equation*}
\end{corollary}

\begin{remark}[Euler characteristic]\label{remark euler char}
Now denote by $\chi_{p,q}$ the Euler characteristic of the complex $(\B^\bullet_{p,q},\delta^\bullet_{p,q})$, as in \cite[Section 2.5]{Ste}, namely
\[
\chi_{p,q}=\sum_{k=0}^{2n-1}(-1)^kh^k_{p,q}.
\]
From the previous considerations it follows that on a given compact complex manifold of complex dimension $n$
\[
\chi_{p,q}=-\chi_{n-q-1,n-p-1}=-\chi_{n-p-1,n-q-1}=\chi_{q,p}.
\]
In fact, using Corollary \ref{corollary hodge numbers} we get
\begin{align*}
\chi_{p,q}&=\sum_{k=0}^{2n-1}(-1)^kh^k_{p,q}\\
&=\sum_{k=0}^{2n-1}(-1)^kh^{2n-k-1}_{n-q-1,n-p-1}\\
&=\sum_{l=0}^{2n-1}(-1)^{l+1}h^l_{n-q-1,n-p-1}\\
&=-\chi_{n-q-1,n-p-1}.
\end{align*}
In particular it follows that $\chi_{p,q}=0$ for $p+q=n-1$.
We provide a graphic explanation of this, when $n=10$, $p=3$ and $q=6$:
\medskip
\begin{center}
\begin{tikzpicture}[scale=0.5]
    \foreach \x in {0,1,...,10}{
    \draw (\x,0)node[below,font=\footnotesize] {\x};
    }
    \foreach \y in {0,1,...,10}{
    \draw (0,\y)node[left,font=\footnotesize] {\y};
    }
    \foreach \x in {0,1,...,10}{
    \foreach \y in {0,1,...,10}{
    \fill[black!30] (\x,\y) circle[radius=1pt];
    }}
    \foreach \x in {0,1,...,3}{
    \foreach \y in {0,1,...,6}{
    \fill[black!80] (\x,\y) circle[radius=2pt];
    }}
    \foreach \x in {4,5,...,10}{
    \foreach \y in {7,8,...,10}{
    \fill[black!80] (\x,\y) circle[radius=2pt];
    }}
\end{tikzpicture}
\\
\begin{tikzpicture}[scale=1]
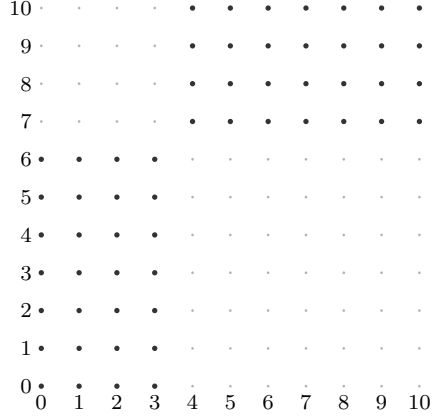
    \draw (0,0)node[below,font=\footnotesize] {Figure 3: the complex $(\B^\bullet_{3,6},\delta^\bullet_{3,6})$ for $n=10$.};
\end{tikzpicture}
\end{center}

Setting $\chi_p:=\sum_{q=0}^n(-1)^qh^{p,q}_\delbar$, where $h^{p,q}_\delbar$ are the Hodge numbers, in \cite[Remark 2.10]{Ste} it is proved that, if $p+q\le n-2$, then
\begin{equation}\label{equation euler bigolin dolbeault}
\chi_{p,q}=\sum_{k=p+1}^{n-q-1}(-1)^{k+1}\chi_k,
\end{equation}
while if $p+q= n-1$ then $\chi_{p,q}=0$. If $p+q\ge n$ then by $\chi_{p,q}=-\chi_{n-q-1,n-p-1}$ there is an analogous relation.
\end{remark}

\begin{proposition}[Inclusion of harmonic forms]
Given a compact Hermitian manifold $(M^n,g)$, for all $0\le k\le 2n-1$ there is an inclusion
\[
\ker \Delta_d\cap\B^k_{p,q}\subseteq\H^k_{p,q},
\]
which is actually an equality for 
\begin{enumerate}
\item $k<\min(p,q)$;
\item $k>n+\max(p,q)$;
\end{enumerate} 
while, for the remaining values of $k$, in general the equality does not hold. Here $\Delta_d$ denotes the Hodge-de Rham Laplacian $dd^*+d^*d$. 
\end{proposition}
\begin{proof}
We know
\[
\ker \Delta_d\cap\B^k_{p,q}=\ker d\cap\ker d^*\cap\B^k_{p,q},
\]
and 
\[
\H^k_{p,q}=\ker\delta^k_{p,q}\cap\ker(\delta^{k-1}_{p,q})^*.
\]
For $0\le k\le p+q$ it holds $(\delta^{k-1}_{p,q})^*=d^*$ and by definition $\ker d\cap\B^k_{p,q}\subseteq\ker\delta^k_{p,q}$, with equality for $k<\min(p,q)$; while for $k>p+q$ it holds $\delta^k_{p,q}=d$ and $\ker d^*\cap\B^k_{p,q}\subseteq\ker(\delta^{k-1}_{p,q})^*$ with equality for $k>n+\max(p,q)$. Counterexamples to the equality for the remaining values of $k$ will follow from explicit computations in Section \ref{sec_iwasawa}. This completes the proof.
\end{proof}

\begin{remark}[Explicit formulation of the Laplacian]\label{rmk laplacian}
We want to obtain a more explicit formulation of the elliptic operators $\square^k_{p,q}$.
Fix $0< p, q\le n$, $p+q\ne 2n$ and $\max(p,q)< k< p+q$. For $l=k,k-1$ we can write
\[
\B^{l}_{p,q}=A^{l-q,q}\oplus\dots\oplus A^{p,l-p}
\]
and
\[
\delta^l_{p,q}=\del\oplus d\oplus\dots\oplus d\oplus\delbar.
\]
Set $\alpha=\alpha^{k-q,q}\oplus\dots\oplus\alpha^{p,k-p}\in\B^{k}_{p,q}$ and compute
\begin{align*}
(\delta^k_{p,q})^*\delta^k_{p,q}\alpha&=d^*\del\alpha^{k-q,q}+d^*d\alpha^{k-q+1,q-1}+\dots\\
&+ d^*d\alpha^{p-1,k-p+1}+d^*\delbar\alpha^{p,k-p}.
\end{align*}
Similarly
\begin{align*}
\delta^{k-1}_{p,q}(\delta^{k-1}_{p,q})^*\alpha&=
\del\del^*\alpha^{k-q,q}+d\delbar^*\alpha^{k-q,q}+dd^*\alpha^{k-q+1,q-1}+\dots\\
&+dd^*\alpha^{p-1,k-p+1}+d\del^*\alpha^{p,k-p}+\delbar\delbar^*\alpha^{p,k-p}.
\end{align*}
Therefore
\begin{align*}
\square^{k}_{p,q}\alpha&=\Delta_\del\alpha^{k-q,q}+\Delta_d\alpha^{k-q+1,q-1}+\dots+\Delta_d\alpha^{p-1,k-p+1}+\Delta_\delbar\alpha^{p,k-p}\\
&+\delbar^*\del\alpha^{k-q,q}+
d\delbar^*\alpha^{k-q,q}+d\del^*\alpha^{p,k-p}+\del^*\delbar\alpha^{p,k-p}.
\end{align*}
\end{remark}

\begin{remark}[Laplacian on a K\"ahler manifold]\label{rmk laplacian kahler}
If the metric is K\"ahler, then by K\"ahler identities $\Delta_d=2\Delta_\del=2\Delta_\delbar$ and $\del\delbar^*+\delbar^*\del=\delbar\del^*+\del^*\delbar=0$. It then follows from Remark \ref{rmk laplacian} that in the same ranges for $p,q,k$
\begin{align*}
\square^{k}_{p,q}\alpha&=\Delta_\delbar\alpha^{k-q,q}+2\Delta_\delbar\alpha^{k-q+1,q-1}+\dots+2\Delta_\delbar\alpha^{p-1,k-p+1}+\Delta_\delbar\alpha^{p,k-p}\\
&+
\delbar\delbar^*\alpha^{k-q,q} +\del\del^*\alpha^{p,k-p}.
\end{align*}
In particular, if the manifold is compact and the metric is K\"ahler, then
\[
\H^k_{p,q}=\H^{k-q,q}_\delbar\oplus\dots\oplus\H^{p,k-p}_\delbar.
\]
If $\min(p,q)\le k\le \max(p,q)$, the computations are similar and yield to the same decomposition of $\H^k_{p,q}$. By Hodge * duality and conjugation we recover the remaining cases, proving the following result.
\end{remark}

\begin{theorem}[Hodge decomposition for compact K\"ahler manifolds]\label{thm_kahler}
Given a compact K\"ahler manifold $(M^n,g)$, for $0\le k\le 2n-1$
\[
\H^k_{p,q}=\ker\Delta_\delbar\cap\B^k_{p,q}.
\]
\end{theorem}

\begin{remark}[Relation with the $\del\delbar$-Lemma]\label{rmk del delbar lemma}
It is possible to show that if, instead of assuming that the manifold is K\"ahler, we just assume that the $\del\delbar$-Lemma holds, then a similar decomposition holds at the  level of cohomology: for $0\le k\le 2n-1$
\[
H^k_{p,q}\simeq H^{\bullet,\bullet}_\delbar\cap\B^k_{p,q}
\]
is induced by the identity, where the intersection with $\B^k_{p,q}$ means that all the representatives in $H^{\bullet,\bullet}_\delbar$ are contained in $\B^k_{p,q}$. This follows, \textit{e.g.}, from the fact that the $\del\delbar$-Lemma holds if and only if the double complex of forms $(A^{\bullet,\bullet},\del,\delbar)$ is a direct sum of dots and squares \cite[Proposition 5.17]{DGMS} and the fact that the cohomology $H^k_{p,q}$ is zero when computed on squares \cite[Lemma 2.3]{Ste}.
\end{remark}

\section{Cohomological relations on compact complex manifolds}\label{section compact}

In Propositions \ref{proposition cohomology} and \ref{proposition bigolin cohomology} we have seen how Bigolin cohomology is related to the other cohomology spaces on a general complex manifold. All these relations follow by only using the definitions. In this section we explore other similar relations under the extra assumption that the manifold is compact. We will need to use the notion of zigzag for which we refer to the in depth treatment \cite{Ste21}.

Given a complex manifold $M^n$ of complex dimension $n$, the bigraded vector space $A^{\bullet,\bullet}:=\oplus_{1\le p,q\le n} A^{p,q}$ endowed with the differentials $\del$ and $\delbar$ has the structure of a double complex $(A^{\bullet,\bullet},\del,\delbar)$. We refer to \cite{Ste21} for the formalisation of the statement that every double complex over a field decomposes into so-called squares and zigzags. Note that de Rham, Dolbeault, $\del$-, Aeppli, Bott-Chern and more in general the $H^k_{p,q}$ cohomology spaces vanish on squares \cite[Section2]{Ste21} \cite[Lemma 2.3]{Ste}, therefore these cohomology spaces depend only on zigzags. 

Following the beginning of \cite[Section 4]{Ste21}, there are some restrictions to the shape of a general zigzag on a compact complex manifold of complex dimension $n$. There are symmetries given by conjugation and duality, and the dots in position $(0,0)$ and $(n,n)$ are as many as the connected components of the manifold. Furthermore, the only zigzags touching the corners $(0,0)$, $(n,0)$, $(0,n)$, $(n,n)$ are dots.

This last property will be referred to as the \lq\lq only dots (and squares) in the corners" rule. The next cohomological relations will follow by this rule.

\begin{proposition}[Cohomological equalities]\label{proposition cohomological equalities}
Given a compact complex manifold $M^n$ of complex dimension $n$, the following identities hold:
\begin{enumerate}
\item if $\min(p,q)=n-1$, then $H^{n-1}_{p,q}= H^{n-1}_{dR}$;
\item if $\max(p,q)=0$, then $H^{n}_{p,q}= H^{n+1}_{dR}$;
\item $H^{n-1}_{0,n-1}=H_\delbar^{0,n-1}$;
\item $H^{n}_{n-1,0}=H_\delbar^{n,1}$;
\item if $0\le p\le n-1$, then $H^{n-1}_{p,n}=H^{n-1}_{p,n-1}$;
\item if $0\le q\le n-1$, then $H^n_{-1,q}=H^{n}_{0,q}$.
\end{enumerate}
Equalities which are similar to (3),(4),(5),(6) hold by conjugation.
\end{proposition}
\begin{proof}
We repeatedly use the fact that the double complex $(A^{\bullet,\bullet},\del,\delbar)$ of a compact complex manifold can be decomposed into squares and zigzags, only zigzags contribute to cohomology, and the only zigzags touching corners $A^{n,0}$ and $A^{0,n}$ are dots.
\end{proof}

We can also use the \lq\lq only dots in the corner" rule to deduce an equality between the spaces of Hodge de Rham harmonic forms and Bigolin harmonic forms.

\begin{proposition}[Equality of harmonic forms]
Given a compact Hermitian manifold $(M^n,g)$, for $k=\min(p,q)=n-1$ or $k=n+\max(p,q)=n$ there is an equality
\[
\ker \Delta_d\cap\B^k_{p,q}=\H^k_{p,q}.
\]
\end{proposition}
\begin{proof}
The first case follows from the fact that the only zigzags touching the corners $A^{n,0}$ and $A^{0,n}$ are dots; the second case then follows by the Hodge $*$ duality of Remark \ref{remark duality}. 
\end{proof}

\begin{remark}[How many invariants?]\label{remark_how_many}
We want to count the minimum number of invariants $h^k_{p,q}$, modulo their symmetries, which are associated to a compact complex manifold of complex dimension $n$ and which are, a priori, different from Aeppli, Bott-Chern, Betti and Dolbeault numbers.

Thanks to the isomorphism \eqref{eq_star_isom}, in order to compute the spaces $\H^k_{p,q}$ for all $p,q,k$, it is sufficient to consider $\H^k_{p,q}$ either for all $0\le p,q\le n$ and $0\le k\le p+q$ or equivalently for all $-1\le p,q\le n-1$ and $p+q< k\le 2n-1$.

We focus on the case $0\le p,q\le n$ and $0\le k\le p+q$, namely we are looking at the bottom left rectangle in the graphic representation. Taking into account the isomorphism on cohomology given by conjugation, we can assume $p\le q$. 

By Proposition \ref{proposition cohomology}, we can exclude the following cases:
\begin{itemize}
\item $k=p+q$, where $H^k_{p,q}$ equals Aeppli cohomology $H^{p,q}_A$;
\item $k<\min(p,q)$, where $H^k_{p,q}$ equals de Rham cohomology $H^k_{dR}$;
\item $p=0$ for $k<q$, where $H^k_{0,q}$ equals to Dolbeault cohomology $H^{0,k}_\delbar$.
%\item $q=0$, corresponding to $\del$-cohomology $H^{k,0}_\del$.
\end{itemize}
We are left with $0< p\le q\le n$, $p+q\ne 2n$, $\min(p,q)=p\le k< p+q$.
Moreover, by Proposition \ref{proposition cohomological equalities}, we can also exclude $k=\min(p,q)=n-1$, where $H^k_{p,q}$ equals de Rham cohomology $H^k_{dR}$.
Therefore, if $p=n-1$ then we are in the case $n\le k$.

Furthermore, by Propositions \ref{proposition bigolin cohomology} and \ref{proposition cohomological equalities},
\begin{align*}
H^k_{p,q}&=H^k_{p,q-1} && \text{for} && p\le k\le q-2,\\
H^{n-1}_{p,n}&=H^{n-1}_{p,n-1} & &\text{for} & & 0\le p\le n-1,
\end{align*}
thus we can exclude also such values for $k$, remaining with:
\begin{align*}
0< p\le q\le n, && p+q\ne 2n, && \max(p,q-1)\le k<p+q,
\end{align*}
and if $\max(p,q-1)=n-1$ then also $n\le k$. In particular, we find no invariants for $n=1,2$ and we focus on $n\ge 3$.

When $0<p=q\le n-2$ the index $k$ can vary from $p$ to $2p-1$, while for $0<p=q=n-1$ then $k$ can vary from $n$ to $2n-3$. When $0<p=q-1\le n-2$ the index $k$ can vary from $p$ to $2p$, while for $0<p=q-1= n-1$ then $k$ can vary from $n$ to $2n-2$.
When $0<p\le q-2\le n-3$ the index $k$ can vary from $q-1$ to $p+q-1$, while for $0<p\le q-2= n-2$ then $k$ can vary from $n$ to $p+n-1$.
Therefore, in order, the count (for $n\ge 3$) reduces to
\begin{align*}
\sum_{1\le p\le n-2}p+n-2&+\sum_{1\le p\le n-2}(p+1)+n-1+\sum_{\substack{1\le p\le n-3\\ p+2\le q\le n-1}}(p+1)+\sum_{1\le p\le n-2}p=\\
&=3n-5+3\sum_{1\le p\le n-1}p+\sum_{\substack{2\le p\le n-2\\ p+1\le q\le n-1}}p\\
&=3n-5+\frac32{(n-1)(n-2)}+\sum_{2\le p\le n-2}p(n-1-p)\\
&=3n-5+\frac32{(n-1)(n-2)}+(n-1)\left(\frac{(n-1)(n-2)}2-1\right)+\\&\ \ -\left(\frac{(n-1)(n-2)(2n-3)}{6}-1\right)\\
&=2n-3+\frac{(n+9)(n-1)(n-2)}6.
\end{align*}
This is the final number of the invariants $h^k_{p,q}$ on a compact complex manifold which are, a priori and modulo their symmetries, interesting to study and different from Betti, Dolbeault, Aeppli and Bott-Chern numbers.
\end{remark}

\section{The double complex of a complex 3-manifold\\ and cohomological invariants}\label{section double complex}

This section contains the main result of the paper: given a compact complex 3-manifold $M^3$, we relate the zigzags of its double complex with the cohomological invariants introduced so far. 

By the discussion of Section \ref{section compact}, apart from the dots in $(0,0)$ whose number is the number of connected components of the manifold, the list of zigzags modulo conjugation and duality on a given compact complex 3-manifold are the following (cf. \cite[proof of Proposition 8.15]{Ste22}).

\pagebreak

The zigzags starting from 1-forms are:
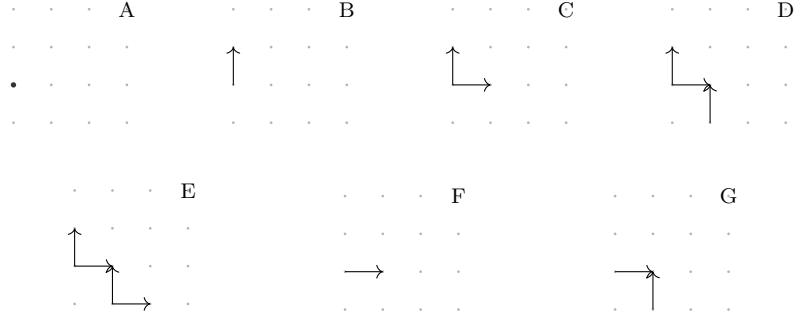
\begin{align*}
\begin{tikzpicture}[scale=0.5]
	\draw (3,3)node[font=\footnotesize]{A};
    \foreach \x in {0,1,...,3}{
    \foreach \y in {0,1,...,3}{
    \fill[black!30] (\x,\y) circle[radius=1pt];
    }}
    \fill[black!80] (0,1) circle[radius=2pt];
\end{tikzpicture}
&&
\begin{tikzpicture}[scale=0.5]
	\draw (3,3)node[font=\footnotesize]{B};
    \foreach \x in {0,1,...,3}{
    \foreach \y in {0,1,...,3}{
    \fill[black!30] (\x,\y) circle[radius=1pt];
    }}
    \draw[->] (0,1) -- (0,2);
\end{tikzpicture}
&&
\begin{tikzpicture}[scale=0.5]
	\draw (3,3)node[font=\footnotesize]{C};
    \foreach \x in {0,1,...,3}{
    \foreach \y in {0,1,...,3}{
    \fill[black!30] (\x,\y) circle[radius=1pt];
    }}
    \draw[->] (0,1) -- (0,2);
    \draw[->] (0,1) -- (1,1);
\end{tikzpicture}
&&
\begin{tikzpicture}[scale=0.5]
	\draw (3,3)node[font=\footnotesize]{D};
    \foreach \x in {0,1,...,3}{
    \foreach \y in {0,1,...,3}{
    \fill[black!30] (\x,\y) circle[radius=1pt];
    }}
    \draw[->] (0,1) -- (0,2);
    \draw[->] (0,1) -- (1,1);
    \draw[->] (1,0) -- (1,1);
\end{tikzpicture}
\end{align*}
\begin{align*}
\begin{tikzpicture}[scale=0.5]
	\draw (3,3)node[font=\footnotesize]{E};
    \foreach \x in {0,1,...,3}{
    \foreach \y in {0,1,...,3}{
    \fill[black!30] (\x,\y) circle[radius=1pt];
    }}
    \draw[->] (0,1) -- (0,2);
    \draw[->] (0,1) -- (1,1);
    \draw[->] (1,0) -- (1,1);
    \draw[->] (1,0) -- (2,0);
\end{tikzpicture}
&&
\begin{tikzpicture}[scale=0.5]
	\draw (3,3)node[font=\footnotesize]{F};
    \foreach \x in {0,1,...,3}{
    \foreach \y in {0,1,...,3}{
    \fill[black!30] (\x,\y) circle[radius=1pt];
    }}
    \draw[->] (0,1) -- (1,1);
\end{tikzpicture}
&&
\begin{tikzpicture}[scale=0.5]
	\draw (3,3)node[font=\footnotesize]{G};
    \foreach \x in {0,1,...,3}{
    \foreach \y in {0,1,...,3}{
    \fill[black!30] (\x,\y) circle[radius=1pt];
    }}
    \draw[->] (0,1) -- (1,1);
    \draw[->] (1,0) -- (1,1);
\end{tikzpicture}
\end{align*}
The zigzags starting from 2-forms are:
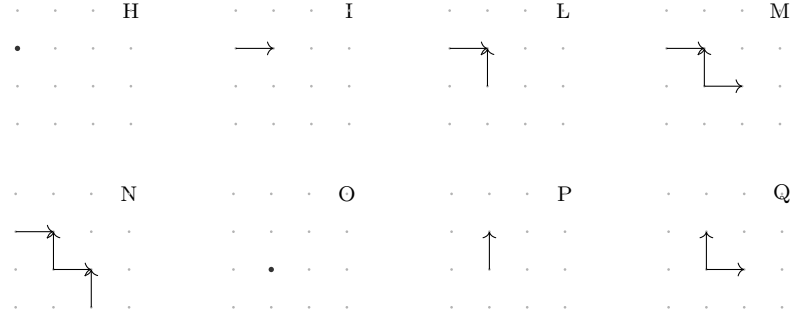
\begin{align*}
\begin{tikzpicture}[scale=0.5]
	\draw (3,3)node[font=\footnotesize]{H};
    \foreach \x in {0,1,...,3}{
    \foreach \y in {0,1,...,3}{
    \fill[black!30] (\x,\y) circle[radius=1pt];
    }}
    \fill[black!80] (0,2) circle[radius=2pt];
\end{tikzpicture}
&&
\begin{tikzpicture}[scale=0.5]
	\draw (3,3)node[font=\footnotesize]{I};
    \foreach \x in {0,1,...,3}{
    \foreach \y in {0,1,...,3}{
    \fill[black!30] (\x,\y) circle[radius=1pt];
    }}
    \draw[->] (0,2) -- (1,2);
\end{tikzpicture}
&&
\begin{tikzpicture}[scale=0.5]
	\draw (3,3)node[font=\footnotesize]{L};
    \foreach \x in {0,1,...,3}{
    \foreach \y in {0,1,...,3}{
    \fill[black!30] (\x,\y) circle[radius=1pt];
    }}
    \draw[->] (0,2) -- (1,2);
    \draw[->] (1,1) -- (1,2);
\end{tikzpicture}
&&
\begin{tikzpicture}[scale=0.5]
	\draw (3,3)node[font=\footnotesize]{M};
    \foreach \x in {0,1,...,3}{
    \foreach \y in {0,1,...,3}{
    \fill[black!30] (\x,\y) circle[radius=1pt];
    }}
    \draw[->] (0,2) -- (1,2);
    \draw[->] (1,1) -- (1,2);
    \draw[->] (1,1) -- (2,1);
\end{tikzpicture}
\end{align*}
\begin{align*}
\begin{tikzpicture}[scale=0.5]
	\draw (3,3)node[font=\footnotesize]{N};
    \foreach \x in {0,1,...,3}{
    \foreach \y in {0,1,...,3}{
    \fill[black!30] (\x,\y) circle[radius=1pt];
    }}
    \draw[->] (0,2) -- (1,2);
    \draw[->] (1,1) -- (1,2);
    \draw[->] (1,1) -- (2,1);
    \draw[->] (2,0) -- (2,1);
\end{tikzpicture}
&&
\begin{tikzpicture}[scale=0.5]
	\draw (3,3)node[font=\footnotesize]{O};
    \foreach \x in {0,1,...,3}{
    \foreach \y in {0,1,...,3}{
    \fill[black!30] (\x,\y) circle[radius=1pt];
    }}
    \fill[black!80] (1,1) circle[radius=2pt];
\end{tikzpicture}
&&
\begin{tikzpicture}[scale=0.5]
	\draw (3,3)node[font=\footnotesize]{P};
    \foreach \x in {0,1,...,3}{
    \foreach \y in {0,1,...,3}{
    \fill[black!30] (\x,\y) circle[radius=1pt];
    }}
    \draw[->] (1,1) -- (1,2);
\end{tikzpicture}
&&
\begin{tikzpicture}[scale=0.5]
	\draw (3,3)node[font=\footnotesize]{Q};
    \foreach \x in {0,1,...,3}{
    \foreach \y in {0,1,...,3}{
    \fill[black!30] (\x,\y) circle[radius=1pt];
    }}
    \draw[->] (1,1) -- (1,2);
    \draw[->] (1,1) -- (2,1);
\end{tikzpicture}
\end{align*}
The zigzags starting from 3-forms which are not dots are dual of the zigzags starting from 2-forms which are not dots, while the dots on 3-forms are:
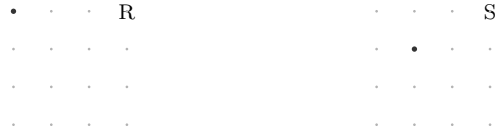
\begin{align*}
\begin{tikzpicture}[scale=0.5]
	\draw (3,3)node[font=\footnotesize]{R};
    \foreach \x in {0,1,...,3}{
    \foreach \y in {0,1,...,3}{
    \fill[black!30] (\x,\y) circle[radius=1pt];
    }}
    \fill[black!80] (0,3) circle[radius=2pt];
\end{tikzpicture}
&&
\begin{tikzpicture}[scale=0.5]
	\draw (3,3)node[font=\footnotesize]{S};
    \foreach \x in {0,1,...,3}{
    \foreach \y in {0,1,...,3}{
    \fill[black!30] (\x,\y) circle[radius=1pt];
    }}
    \fill[black!80] (1,2) circle[radius=2pt];
\end{tikzpicture}
\end{align*}
All the other zigzags can be obtained by conjugation or duality, or are forbidden by the \lq\lq only dots and squares in the corners" rule. We provide a graphic explanation of zigzags obtained by conjugation or duality:

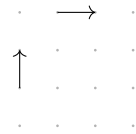
\begin{align*}
\text{B and its conjugate} && \text{B and its dual}\\
\begin{tikzpicture}[scale=0.5]
	\draw (3,3)node[font=\footnotesize]{};
    \foreach \x in {0,1,...,3}{
    \foreach \y in {0,1,...,3}{
    \fill[black!30] (\x,\y) circle[radius=1pt];
    }}
    \draw[->] (0,1) -- (0,2);
    \draw[->] (1,0) -- (2,0);
\end{tikzpicture}
&&
\begin{tikzpicture}[scale=0.5]
	\draw (3,3)node[font=\footnotesize]{};
    \foreach \x in {0,1,...,3}{
    \foreach \y in {0,1,...,3}{
    \fill[black!30] (\x,\y) circle[radius=1pt];
    }}
    \draw[->] (0,1) -- (0,2);
    \draw[->] (1,3) -- (2,3);
\end{tikzpicture}
\end{align*}

The double complex $(A^{\bullet,\bullet},\del,\delbar)$ of a compact complex 3-manifold is the direct sum of possibly infinitely many squares (which do not contribute to cohomology) and finitely many zigzags as above (or their conjugates or duals). We denote the above zigzags by the letters A,B,$\dots$,S as previously indicated. By an abuse of notation, we also refer to the multiplicities of the respective zigzags inside the direct sum decomposition of the double complex $(A^{\bullet,\bullet},\del,\delbar)$ by the same letters A,B,$\dots$,S.

\begin{remark}[Cohomological invariants via zigzags]\label{rmk cohom zigzags}
By \cite[Proposition 6, Lemma 8]{Ste21}, we know which zigzags contribute to de Rham, Aeppli, Bott-Chern, Dolbeault and $\del$-cohomology. Namely,
\begin{align*}
h^{0,1}_\delbar&=A+D+F+G,\\
h^{0,1}_\del&=A+B,\\
h^{0,1}_{BC}&=A,\\
h^{0,1}_A&=A+B+C+2D+E+F+G,\\
b^{1}&=2A+G,\\
h^{0,2}_\delbar&=H+I+L+M+N=h^{0,2}_A,\\
h^{0,2}_\del&=B+C+D+E+H=h^{0,2}_{BC},\\
h^{1,1}_\delbar&=C+F+L+O+P,\\
h^{1,1}_{BC}&=2C+2D+E+2F+G+O,\\
h^{1,1}_A&=2L+2M+N+O+2P+Q,\\
b^{2}&=2C+E+2H+2L+N+O,\\
h^{0,3}_\delbar&=R=h^{0,3}_\del=h^{0,3}_{BC}=h^{0,3}_A,\\
h^{1,2}_\delbar&=I+M+P+Q+S=h^{1,2}_\del,\\
h^{1,2}_{BC}&=I+L+2M+N+P+Q+S=h^{1,2}_A,\\
b^{3}&=2R+2S+2Q.
\end{align*}
The other Aeppli, Bott-Chern, Dolbeault, $\del$- and Betti numbers are obtained via conjugation or duality:
\begin{align*}
h^{p,q}_A&=h^{q,p}_A, &  h^{p,q}_{BC}&=h^{q,p}_{BC}, & h^{p,q}_\delbar&=h^{q,p}_\del,\\
h^{p,q}_A&=h^{3-p,3-q}_{BC}, & h^{p,q}_\delbar&=h^{3-p,3-q}_\delbar, & b^k&=b^{6-k}.
\end{align*}
The above equalities can also be checked by hand. E.g., let us compute $h^{1,2}_\delbar$: the only zigzags starting from 2-forms and touching $(1,2)$ are I,L,M,N,P,Q, but L,M,N,P,Q correspond to $\delbar$-exact forms, and the only zigzags starting from 3-forms and touching $(1,2)$ which correspond to $\delbar$-closed forms are S, the duals of P,Q and the dual of the conjugate of M. Also note that $b^3$ is even, which is true on any compact smooth manifold of real dimension 6.

Placing all the above linear expressions in matrix form, we obtain a 15x17 matrix of rank 14. In this way we re-obtain the well known Fr\"olicher relations in the particular case of complex 3-manifolds.
\end{remark}

\begin{corollary}[Fr\"olicher relations]\label{corollary relation betti dolbeault}
Given a compact complex 3-manifold, if we set $h^k_\delbar:=\sum_{p+q=k}h^{p,q}_\delbar$, then
\[
\sum_{k=0}^6(-1)^kb^k=\sum_{k=0}^6(-1)^kh^k_\delbar,
\]
and for all $k$
\[
b^k\le h^k_\delbar.
\]
\end{corollary}
\begin{proof}
By duality and $b^0=h^0_\delbar$, the first claim is equivalent to show
\[
b^3-h^3_\delbar=2b^2-2h^2_\delbar+2h^1_\delbar-2b^1.
\]
It is then enough to check that both sides of the above equality coincide with the multiplicities $-2I-2M-2P$. The second claim is immediate to check.
\end{proof}

\begin{remark}[The Bigolin cohomology $H^k_{p,q}$ via zigzags]\label{rmk cohom bigolin zigzags}
By Remark \ref{remark_how_many}, up to symmetries there are 7 invariants $h^k_{p,q}$ to consider on a complex 3-manifold,
which we list drawing the pertinent part of the complex $(\B^\bullet_{p,q},\delta^\bullet_{p,q})$:
\begin{align*}
 h^1_{1,1} & & h^1_{1,2},h^2_{1,2} & & h^3_{1,3} \\
\begin{tikzpicture}[scale=0.5]
    \foreach \x in {0,1,...,3}{
    \draw (\x,0)node[below,font=\footnotesize] {\x};
    }
    \foreach \y in {0,1,...,3}{
    \draw (0,\y)node[left,font=\footnotesize] {\y};
    }
    \foreach \x in {0,1,...,3}{
    \foreach \y in {0,1,...,3}{
    \fill[black!30] (\x,\y) circle[radius=1pt];
    }}
    \foreach \x in {0,1,...,1}{
    \foreach \y in {0,1,...,1}{
    \fill[black!80] (\x,\y) circle[radius=2pt];
    }}
\end{tikzpicture}
& &
\begin{tikzpicture}[scale=0.5]
    \foreach \x in {0,1,...,3}{
    \draw (\x,0)node[below,font=\footnotesize] {\x};
    }
    \foreach \y in {0,1,...,3}{
    \draw (0,\y)node[left,font=\footnotesize] {\y};
    }
    \foreach \x in {0,1,...,3}{
    \foreach \y in {0,1,...,3}{
    \fill[black!30] (\x,\y) circle[radius=1pt];
    }}
    \foreach \x in {0,1,...,1}{
    \foreach \y in {0,1,...,2}{
    \fill[black!80] (\x,\y) circle[radius=2pt];
    }}
\end{tikzpicture}
& &
\begin{tikzpicture}[scale=0.5]
    \foreach \x in {0,1,...,3}{
    \draw (\x,0)node[below,font=\footnotesize] {\x};
    }
    \foreach \y in {0,1,...,3}{
    \draw (0,\y)node[left,font=\footnotesize] {\y};
    }
    \foreach \x in {0,1,...,3}{
    \foreach \y in {0,1,...,3}{
    \fill[black!30] (\x,\y) circle[radius=1pt];
    }}
    \foreach \x in {0,1,...,1}{
    \foreach \y in {0,1,...,3}{
    \fill[black!80] (\x,\y) circle[radius=2pt];
    }}
\end{tikzpicture}
\end{align*}

\begin{align*}
 h^3_{2,2} & & h^3_{2,3},h^4_{2,3}\\
\begin{tikzpicture}[scale=0.5]
    \foreach \x in {0,1,...,3}{
    \draw (\x,0)node[below,font=\footnotesize] {\x};
    }
    \foreach \y in {0,1,...,3}{
    \draw (0,\y)node[left,font=\footnotesize] {\y};
    }
    \foreach \x in {0,1,...,3}{
    \foreach \y in {0,1,...,3}{
    \fill[black!30] (\x,\y) circle[radius=1pt];
    }}
    \foreach \x in {0,1,...,2}{
    \foreach \y in {0,1,...,2}{
    \fill[black!80] (\x,\y) circle[radius=2pt];
    }}
\end{tikzpicture}
& &
\begin{tikzpicture}[scale=0.5]
    \foreach \x in {0,1,...,3}{
    \draw (\x,0)node[below,font=\footnotesize] {\x};
    }
    \foreach \y in {0,1,...,3}{
    \draw (0,\y)node[left,font=\footnotesize] {\y};
    }
    \foreach \x in {0,1,...,3}{
    \foreach \y in {0,1,...,3}{
    \fill[black!30] (\x,\y) circle[radius=1pt];
    }}
    \foreach \x in {0,1,...,2}{
    \foreach \y in {0,1,...,3}{
    \fill[black!80] (\x,\y) circle[radius=2pt];
    }}
\end{tikzpicture}
\end{align*}
These invariants can be computed using zigzags as
\begin{align*}
h^1_{1,1}&=2A+2B+2D+E+G,\\
h^1_{1,2}&=2A+B+D+G,\\
h^2_{1,2}&=C+H+2L+M+N+O+P,\\
h^3_{1,3}&=M+P+Q+R+S,\\
h^3_{2,2}&=2I+2M+N+2Q+2S,\\
h^3_{2,3}&=I+M+2Q+R+2S,\\
h^4_{2,3}&=2C+D+E+F+H+L+O.
\end{align*}
We illustrate these equalities:
\begin{itemize}
\item[$h^1_{1,1}$:]
The zigzags starting on 1-forms which correspond to $\delta^1_{1,1}$-closed forms are A,B,D,E,G, plus the conjugates of A,B,D (there are no $\delta^0_{1,1}$-exact forms by the \lq\lq only dots and squares in the corners" rule). Therefore A,B,D are counted twice while E,G once.
\item[$h^1_{1,2}$:]
The zigzags starting on 1-forms which correspond to $\delta^1_{1,2}$-closed forms are A,G, plus the conjugates of A,B,D.
\item[$h^2_{1,2}$:]
All the zigzags starting on 1-forms correspond to $\delta^2_{1,2}$-closed 2-forms but all except for one of these forms are $\delta^1_{1,2}$-exact. For example, the forms corresponding to $(0,2)$ and $(1,1)$ in D are $\delta^1_{1,2}$-exact: write $\alpha^{0,2}\in A^{0,2}$ and $\alpha^{1,1}\in A^{1,1}$ for these forms, then
\begin{align*}
\alpha^{0,2}=\delbar\beta^{0,1}, && \alpha^{1,1}=\del\beta^{0,1}=\delbar\beta^{1,0},
\end{align*}
for some $\beta^{0,1}\in A^{1,0}$ and $\beta^{1,0}\in A^{1,0}$, and this can be rewritten as
\begin{align*}
\alpha^{0,2}+\alpha^{1,1}=\delta^1_{1,2}\beta^{0,1}, &&\alpha^{1,1}=\delta^1_{1,2}\beta^{1,0}.
\end{align*}
Similarly, also B,E,F,G and the conjugates of C,D,F correspond to $\delta^1_{1,2}$-exact forms. On the other hand, if we denote by $\gamma^{0,2}\in A^{0,2}$ and $\gamma^{1,1}\in A^{1,1}$ the forms corresponding to $(0,2)$ and $(1,1)$ in C, then
\begin{align*}
\gamma^{0,2}=\delbar\beta^{0,1}, && \gamma^{1,1}=\del\beta^{0,1},
\end{align*}
therefore
\[
\gamma^{0,2}+\gamma^{1,1}=\delta^1_{1,2}\beta^{0,1}.
\]
Summing up, C provides two linearly independent $\delta^2_{1,2}$-closed 2-forms but their sum is $\delta^1_{1,2}$-exact, thus C is counted once in $h^2_{1,2}$.

The zigzags starting on 2-forms corresponding to $\delta^2_{1,2}$-closed forms are H,L,M,N,O, plus the conjugates of L,P.
\item[$h^3_{1,3}$:] 
All the zigzags starting on 2-forms correspond to $\delta^2_{1,3}$-exact forms. On the other hand, among the zigzags starting on 3-forms, R,S and the duals of P,Q correspond to $\delta^3_{1,3}$-closed forms.
\item[$h^3_{2,2}$:]
Among the zigzags starting on 2-forms corresponding to $\delta^3_{2,2}$-closed 3-forms, only Q contributes (once) to $h^3_{2,2}$, similarly to the case of $C$ already described for $h^2_{1,2}$. Moreover, the zigzags starting on 3-forms corresponding to $\delta^3_{2,2}$-closed forms are: S and its conjugate, the dual of I and its conjugate, the dual of M and its conjugate, the dual of N, the dual of Q.
\item[$h^3_{2,3}$:]
Similarly to the previous case, among the zigzags starting on 2-forms corresponding to $\delta^3_{2,3}$-closed 3-forms, only Q contributes (once) to $h^3_{2,3}$. Furthermore, the zigzags starting on 3-forms corresponding to $\delta^3_{2,3}$-closed forms are: R, S and its conjugate, the dual of Q, the dual of the conjugate of I.
\item[$h^4_{2,3}$:]
Among the zigzags starting on 3-forms corresponding to $\delta^4_{2,3}$-closed 4-forms, only the dual of L contributes (once) to $h^4_{2,3}$, similarly to the case of $C$ already described for $h^2_{1,2}$. On the other hand, the zigzags starting on 4-forms corresponding to $\delta^4_{2,3}$-closed forms are the duals of H,O,C,D,E and the conjugates of the duals of C,F.
\end{itemize}
Actually, we can generalise the contributions of the zigzags to any cohomology space $H^k_{p,q}$ as follows, in a similar fashion of \cite[proof of Proposition 6]{Ste21}. The proof is the same and will be omitted, and the idea is essentially contained in the above computations. Every zigzag can be seen as a set of dots $(r,s)$. We define the \emph{restriction} of a zigzag Z to the space $\B^\bullet_{p,q}$ as a new zigzag Z' which is given by all the dots in Z corresponding to thick points in the graphic representation of $\B^\bullet_{p,q}$ (cf. Remark \ref{remark_graphic_rep}). E.g., B is the set $\{(0,1),(0,2)\}$, and $B$ restricted to $\B^\bullet_{1,1}$ is A, i.e., $\{(0,1)\}$. We recall that a zigzag is \emph{even} or \emph{odd} if its cardinality is even or odd. We say that an odd zigzag has most components in degree $k$ if most of its elements $(r,s)$ satisfy $r+s=k$.
\begin{proposition}[Contribution of zigzags to $h^k_{p,q}$]\label{prop zigzag contribute to bigolin}
Given a compact complex $n$-manifold and $k\le p+q$, $h^k_{p,q}$ equals the sum of the multiplicities of the zigzags which, once they have been restricted to $\B^\bullet_{p,q}$, are odd zigzags with most components in degree $k$.
\end{proposition}
\end{remark}

We are now ready to prove our main result.

\begin{theorem}[Bijection between zigzags and cohomological invariants]\label{thm zigzags 3manifolds}
Given a compact complex 3-manifold, the multiplicities of the zigzags in the direct sum decomposition of the double complex $(A^{\bullet,\bullet},\del,\delbar)$ are completely characterised by the invariants $h^{p,q}_\delbar$, $h^{p,q}_A$, $b^k$ and $h^k_{p,q}$ for all $p,q,k$.
\end{theorem}
\begin{proof}
\begin{table}[!ht]
\captionof{table}{Contribution of zigzags to cohomology}%:\\ \url{https://docs.google.com/spreadsheets/d/1YSWhZ7wOH0HdUKqJRkzvWaSbJUYBvCMipzbAwoF00SY/edit?usp=sharing}}
\label{table 22x17}
    \centering
    \begin{tabular}{|l|l|l|l|l|l|l|l|l|l|l|l|l|l|l|l|l|l|}
    \hline
         &     A & B & C & D & E & F & G & H & I & L & M & N & O & P & Q & R & S \\ \hline
        $h^{0,1}_\delbar$ & 1 & 0 & 0 & 1 & 0 & 1 & 1 & 0 & 0 & 0 & 0 & 0 & 0 & 0 & 0 & 0 & 0 \\ \hline
        $h^{0,1}_\del$ & 1 & 1 & 0 & 0 & 0 & 0 & 0 & 0 & 0 & 0 & 0 & 0 & 0 & 0 & 0 & 0 & 0 \\ \hline
        $h^{0,1}_{BC}$ & 1 & 0 & 0 & 0 & 0 & 0 & 0 & 0 & 0 & 0 & 0 & 0 & 0 & 0 & 0 & 0 & 0 \\ \hline
        $h^{0,1}_A$ & 1 & 1 & 1 & 2 & 1 & 1 & 1 & 0 & 0 & 0 & 0 & 0 & 0 & 0 & 0 & 0 & 0 \\ \hline
        $b^1$ &     2 & 0 & 0 & 0 & 0 & 0 & 1 & 0 & 0 & 0 & 0 & 0 & 0 & 0 & 0 & 0 & 0 \\ \hline
        $h^1_{1,1}$ & 2 & 2 & 0 & 2 & 1 & 0 & 1 & 0 & 0 & 0 & 0 & 0 & 0 & 0 & 0 & 0 & 0 \\ \hline
        $h_{1,2}^1$ & 2 & 1 & 0 & 1 & 0 & 0 & 1 & 0 & 0 & 0 & 0 & 0 & 0 & 0 & 0 & 0 & 0 \\ \hline
        $h^{0,2}_\delbar$ & 0 & 0 & 0 & 0 & 0 & 0 & 0 & 1 & 1 & 1 & 1 & 1 & 0 & 0 & 0 & 0 & 0 \\ \hline
        $h^{0,2}_\del$ & 0 & 1 & 1 & 1 & 1 & 0 & 0 & 1 & 0 & 0 & 0 & 0 & 0 & 0 & 0 & 0 & 0 \\ \hline
        $b^2$ &     0 & 0 & 2 & 0 & 1 & 0 & 0 & 2 & 0 & 2 & 0 & 1 & 1 & 0 & 0 & 0 & 0 \\ \hline
        $h^{1,1}_\delbar$ & 0 & 0 & 1 & 0 & 0 & 1 & 0 & 0 & 0 & 1 & 0 & 0 & 1 & 1 & 0 & 0 & 0 \\ \hline
        $h^{1,1}_{BC}$ & 0 & 0 & 2 & 2 & 1 & 2 & 1 & 0 & 0 & 0 & 0 & 0 & 1 & 0 & 0 & 0 & 0 \\ \hline
        $h^{1,1}_A$ & 0 & 0 & 0 & 0 & 0 & 0 & 0 & 0 & 0 & 2 & 2 & 1 & 1 & 2 & 1 & 0 & 0 \\ \hline
        $h_{1,2}^2$ & 0 & 0 & 1 & 0 & 0 & 0 & 0 & 1 & 0 & 2 & 1 & 1 & 1 & 1 & 0 & 0 & 0 \\ \hline
        $h^{0,3}_\delbar$ & 0 & 0 & 0 & 0 & 0 & 0 & 0 & 0 & 0 & 0 & 0 & 0 & 0 & 0 & 0 & 1 & 0 \\ \hline
        $h^{1,2}_\delbar$ & 0 & 0 & 0 & 0 & 0 & 0 & 0 & 0 & 1 & 0 & 1 & 0 & 0 & 1 & 1 & 0 & 1 \\ \hline
        $h^{1,2}_{BC}$ & 0 & 0 & 0 & 0 & 0 & 0 & 0 & 0 & 1 & 1 & 2 & 1 & 0 & 1 & 1 & 0 & 1 \\ \hline
        $b^3$ &     0 & 0 & 0 & 0 & 0 & 0 & 0 & 0 & 0 & 0 & 0 & 0 & 0 & 0 & 2 & 2 & 2 \\ \hline
        $h_{1,3}^3$ & 0 & 0 & 0 & 0 & 0 & 0 & 0 & 0 & 0 & 0 & 1 & 0 & 0 & 1 & 1 & 1 & 1 \\ \hline
        $h_{2,2}^3$ & 0 & 0 & 0 & 0 & 0 & 0 & 0 & 0 & 2 & 0 & 2 & 1 & 0 & 0 & 2 & 0 & 2 \\ \hline
        $h_{2,3}^3$ & 0 & 0 & 0 & 0 & 0 & 0 & 0 & 0 & 1 & 0 & 1 & 0 & 0 & 0 & 2 & 1 & 2 \\ \hline
        $h_{2,3}^4$ & 0 & 0 & 2 & 1 & 1 & 1 & 0 & 1 & 0 & 1 & 0 & 0 & 1 & 0 & 0 & 0 & 0 \\ \hline
    \end{tabular}
\end{table}
By Remarks \ref{rmk cohom zigzags} and \ref{rmk cohom bigolin zigzags}, for all $p,q,k$ we know the linear expressions of the invariants $h^{p,q}_\delbar$, $h^{p,q}_A$, $b^k$ and $h^k_{p,q}$ in functions of the multiplicities A,B,$\dots$,S of the zigzags in the direct sum decomposition of the double complex $(A^{\bullet,\bullet},\del,\delbar)$. We write these linear expressions in matrix form in Table \ref{table 22x17}. We obtain a 22x17 matrix $T$ whose rank is 17, proving the claim.
\end{proof}

\begin{remark}[Contribution of cohomological invariants to zigzags]\label{remark inverse}
By Remark \ref{remark_how_many} there are 7 invariants $h^k_{p,q}$ providing information in complex dimension 3. By Corollary \ref{corollary relation betti dolbeault}, the 15x17 matrix $V$ obtained by removing the 7 rows corresponding to $h^k_{p,q}$ from the 22x17 matrix $T$ of Table \ref{table 22x17} has rank 14. Therefore, at least 3 invariants $h^k_{p,q}$ are a priori needed to deduce the multiplicities of all zigzags. We choose to consider $h^1_{1,1}$ and $h^1_{1,2}$ for low degree reasons, then only $h^3_{2,3}$ or $h^4_{2,3}$ provides the maximum rank, and among those we choose $h^4_{2,3}$.

Summing up, we remove from $T$ the rows corresponding to $b^3$, $h^2_{1,2}$, $h^3_{1,3}$, $h^3_{2,2}$, $h^3_{2,3}$, obtaining a 17x17 matrix $U$ of rank 17. 
This matrix is invertible with inverse given by Table \ref{table inverse}. This table can be used to write down all the zigzags of a double complex, knowing Betti (except for $b^3$), Dolbeault and Aeppli numbers plus Bigolin numbers $h^1_{1,1}$, $h^1_{1,2}$ and $h^4_{2,3}$. One only has to multiply the matrix $U^{-1}$ with a vector containing the requested cohomological invariants. This is done in Remark \ref{rmk flavi}.

\begin{table}[!ht]
\captionof{table}{Contribution of cohomology to zigzags}\label{table inverse}
    \centering
    \rotatebox{270}{
    \begin{tabular}{|l|l|l|l|l|l|l|l|l|l|l|l|l|l|l|l|l|l|}
    \hline
         & $h^{0,1}_\delbar$ & $h^{0,1}_\del$ & $h^{0,1}_{BC}$ & $h^{0,1}_A$ & $b^1$ & $h_{1,1}^1$ & $h_{1,2}^1$ & $h^{0,2}_\delbar$ & $h^{0,2}_\del$ & $b^2$ & $h^{1,1}_\delbar$ & $h^{1,1}_{BC}$ & $h^{1,1}_A$ & $h^{0,3}_\delbar$ & $h^{1,2}_\delbar$ & $h^{1,2}_{BC}$ & $h_{2,3}^4$ \\ \hline
        A & 0 & 0 & 1 & 0 & 0 & 0 & 0 & 0 & 0 & 0 & 0 & 0 & 0 & 0 & 0 & 0 & 0 \\ \hline
        B & 0 & 1 & -1 & 0 & 0 & 0 & 0 & 0 & 0 & 0 & 0 & 0 & 0 & 0 & 0 & 0 & 0 \\ \hline
        C & -1 & 0 & 0 & 1 & 0 & -1 & 1 & 0 & 0 & 0 & 0 & 0 & 0 & 0 & 0 & 0 & 0 \\ \hline
        D & 0 & -1 & 1 & 0 & -1 & 0 & 1 & 0 & 0 & 0 & 0 & 0 & 0 & 0 & 0 & 0 & 0 \\ \hline
        E & 0 & 0 & 0 & 0 & 1 & 1 & -2 & 0 & 0 & 0 & 0 & 0 & 0 & 0 & 0 & 0 & 0 \\ \hline
        F & 1 & 1 & 0 & 0 & 0 & 0 & -1 & 0 & 0 & 0 & 0 & 0 & 0 & 0 & 0 & 0 & 0 \\ \hline
        G & 0 & 0 & -2 & 0 & 1 & 0 & 0 & 0 & 0 & 0 & 0 & 0 & 0 & 0 & 0 & 0 & 0 \\ \hline
        H & 1 & 0 & 0 & -1 & 0 & 0 & 0 & 0 & 1 & 0 & 0 & 0 & 0 & 0 & 0 & 0 & 0 \\ \hline
        I & -1 & 0 & 0 & 1 & 0 & 0 & 0 & 1 & -1 & 0 & 0 & 0 & 0 & 0 & 1 & -1 & 0 \\ \hline
        L & 0 & 0 & -1 & 1 & 0 & 0 & 0 & 0 & -1 & 0 & 0 & -1 & 0 & 0 & 0 & 0 & 1 \\ \hline
        M & 0 & 0 & -1 & -1 & 1 & 0 & 0 & 0 & 1 & -1 & 0 & 0 & 0 & 0 & -1 & 1 & 1 \\ \hline
        N & 0 & 0 & 2 & 0 & -1 & 0 & 0 & 0 & 0 & 1 & 0 & 1 & 0 & 0 & 0 & 0 & -2 \\ \hline
        O & 0 & 0 & 0 & -2 & 0 & 1 & 0 & 0 & 0 & 0 & 0 & 1 & 0 & 0 & 0 & 0 & 0 \\ \hline
        P & 0 & -1 & 1 & 0 & 0 & 0 & 0 & 0 & 1 & 0 & 1 & 0 & 0 & 0 & 0 & 0 & -1 \\ \hline
        Q & 0 & 2 & 0 & 2 & -1 & -1 & 0 & 0 & -2 & 1 & -2 & 0 & 1 & 0 & 2 & -2 & 0 \\ \hline
        R & 0 & 0 & 0 & 0 & 0 & 0 & 0 & 0 & 0 & 0 & 0 & 0 & 0 & 1 & 0 & 0 & 0 \\ \hline
        S & 1 & -1 & 0 & -2 & 0 & 1 & 0 & -1 & 1 & 0 & 1 & 0 & -1 & 0 & -1 & 2 & 0 \\ \hline
    \end{tabular}
    }
\end{table}
\end{remark}

\begin{remark}[Necessity of Bigolin numbers $h^k_{p,q}$ in Theorem \ref{thm zigzags 3manifolds}]\label{remark necessity Bigolin}
By \cite{A} there are complex structures on the same smooth 6-manifold with the same numbers $h^{p,q}_\delbar$ (and clearly $b^k$) but different $h^{p,q}_A$. Therefore, we already know that the statement of Theorem \ref{thm zigzags 3manifolds} is false if we remove $h^{p,q}_A$ and $h^k_{p,q}$ from it. 

Actually, we can show that the statement of Theorem \ref{thm zigzags 3manifolds} is false if we remove $h^k_{p,q}$ from it. By \cite[Section 9]{Ste22}, for every zigzag A,B,C,E,$\dots$,S (all except for D, for which the result is not known) there is a compact complex manifold where that zigzag occurs in the double complex. Taking the disjoint union of these manifolds, we obtain a compact complex manifold $M$ which is not connected where all the 16 zigzags A,B,C,E,$\dots$,S occur in the double complex while D does not occur. Removing the column corresponding to D and the row corresponding to $h^1_{1,2}$ from the 17x17 matrix $U$ of Remark \ref{remark inverse}, we are left with a 16x16 matrix $U_1$ whose rank is 16. In particular, knowing only the invariants $h^{p,q}_\delbar$, $h^{p,q}_A$, $b^k$ is not sufficient to compute all the multiplicities A,B,C,E,$\dots$,S. Therefore this manifold $M$ provides the counterexample we need.

More is true. In fact, removing from $U_1$ the rows corresponding to $h^1_{1,1}$ and $h^4_{2,3}$ (this is the same as removing from the matrix $V$ of Remark \ref{remark inverse} the row corresponding to $b^3$ and the column corresponding to $D$), we obtain a 14x16 matrix $U_2$ of rank 14. The matrix $U_2$ represents a linear map sending a vector with multiplicities A,B,C,E,$\dots$,S to Betti, Dolbeault and Aeppli numbers. Its kernel is 2-dimensional, therefore again by \cite[Section 9]{Ste22} we can construct two compact complex manifolds with same Betti, Dolbeault and Aeppli numbers but different Bigolin numbers.

The above problems remain open for connected compact complex manifolds.
\end{remark}

\begin{remark}[Bigolin numbers $h^k_{p,q}$ alone do not fully describe zigzags - Necessity of Hodge numbers $h^{p,q}_\delbar$ in Theorem \ref{thm zigzags 3manifolds}]\label{remark necessity of hodge numbers}
From Section \ref{section_complex} we know that Aeppli (and dually Bott-Chern) cohomology $H^{p,q}_A$, de Rham cohomology $H^k_{dR}$, and the particular cases of Dolbeault and $\del$-cohomologies $H^{0,k}_\delbar$ and $H^{k,0}_\del$ can be seen as particular cases of Bigolin cohomology $H^k_{p,q}$. One could therefore ask if Bigolin numbers $h^k_{p,q}$ alone are sufficient to characterise the multiplicities of all zigzags in a similar fashion of Theorem \ref{thm zigzags 3manifolds}. Actually, this is equivalent to asking if the statement of Theorem \ref{thm zigzags 3manifolds} continues to hold if we remove Hodge numbers $h^{p,q}_\delbar$ from it. To address such a problem, we remove from the 22x17 matrix $T$ of Table \ref{table 22x17} the rows corresponding to $h^{0,1}_\del$, $h^{1,1}_\delbar$ and $h^{1,2}_\delbar$, which are the only invariants in the matrix not obtainable as $h^k_{p,q}$, finding a 19x17 matrix of rank 15. The same manifold $M$ of Remark \ref{remark necessity Bigolin} then provides the counterexample we need, since all the 16 zigzags A,B,C,E,$\dots$,S occur in its double complex.
\end{remark}

\begin{remark}
Restricted to complex dimension 3, \cite[Problem 11.2]{Ste22} asks to construct a compact complex manifold of complex dimension 3 where the zigzag $D$ appears in the direct sum decomposition of the double complex. Interpreting $D$ as a multiplicity, this is equivalent to $D>0$. We can use Table \ref{table inverse} to read
\[
D=-h^{0,1}_\del+h^{0,1}_{BC}-b^1+h^1_{1,2},
\]
which provides an equivalent rephrasing of the above problem.
\end{remark}

We end this section by noticing that the five linear relations between $h^{p,q}_\delbar$, $h^{p,q}_A$, $b^k$ and $h^k_{p,q}$ to be expected from the rank of the 22x17 matrix $T$ being 17 all come from equalities bewteen Euler characteristics.

\begin{theorem}\label{thm linear relations bigolin euler}
Given a compact complex 3-manifold, the linear expressions describing $h^2_{1,2}$, $h^3_{1,3}$, $h^3_{2,2}$ and $h^3_{2,3}$ in functions of other invariants $h^{p,q}_\delbar$, $h^{p,q}_A$, $b^k$ and $h^k_{p,q}$ which can be obtained from Tables \ref{table 22x17} and \ref{table inverse} are precisely the ones given by equation \eqref{equation euler bigolin dolbeault}.
\end{theorem}
\begin{proof}
The 17x17 matrix $U^{-1}$ of Table \ref{table inverse} expresses the contribution of 17 cohomological invariants to the multiplicities of zigzags.
Call $T_1$ the 4x17 matrix which is obtained from the 22x17 matrix $T$ of Table \ref{table 22x17} by removing all the rows except for the ones corresponding to $h^2_{1,2}$, $h^3_{1,3}$, $h^3_{2,2}$ and $h^3_{2,3}$.

If we multiply $T_1$ by $U^{-1}$ we find a 4x17 matrix expressing $h^2_{1,2}$, $h^3_{1,3}$, $h^3_{2,2}$ and $h^3_{2,3}$ in functions of the other cohomological invariants.
We report these linear expressions:
\begin{align*}
h^2_{1,2}&=-h^{0,1}_\del-h^{0,1}_A+h^1_{1,2}+h^{0,2}_\del+h^{1,1}_\delbar-h^{1,2}_\delbar+h^{1,2}_{BC},\\
h^3_{1,3}&=h^{0,1}_\delbar-h^{0,1}_A-h^{0,2}_\delbar+h^{0,2}_\del+h^{0,3}_\delbar+h^{1,2}_{BC},\\
h^3_{2,2}&=2h^{0,1}_\del-b^1-2h^{0,2}_\del+b^2-2h^{1,1}_\delbar+h^{1,1}_{BC}+2h^{1,2}_\delbar,\\
h^3_{2,3}&=h^{0,1}_\delbar+2h^{0,1}_\del-h^{0,1}_{BC}-b^1-h^{0,2}_\delbar-2h^{0,2}_\del+b^2-2h^{1,1}_\delbar+h^{0,3}_{\delbar}+2h^{1,2}_\delbar+h^4_{2,3}.
\end{align*}
These four linear expressions are equivalent to
\begin{align*}
\chi_{1,2}&=-\chi_1,\\
\chi_{1,3}&=\chi_0-\chi_1,\\
\chi_{2,2}&=-\chi_1+\chi_2,\\
\chi_{2,3}&=\chi_0-\chi_1+\chi_2.
\end{align*}
which is precisely \eqref{equation euler bigolin dolbeault} after the identities $\chi_{1,2}=-\chi_{0,1}$, $\chi_{1,3}=-\chi_{-1,1}$, $\chi_{2,2}=-\chi_{0,0}$ and $\chi_{2,3}=-\chi_{-1,0}$ of Remark \ref{remark euler char}.
\end{proof}

\begin{remark}
The 22x17 matrix $T$ relating multiplicities of zigzags to cohomological invariants has rank 17. The five linear relations between $h^{p,q}_\delbar$, $h^{p,q}_A$, $b^k$ and $h^k_{p,q}$ to be expected from the rank being 17 all come from equalities bewteen Euler characteristics: one is given by Corollary \ref{corollary relation betti dolbeault} and four are given by Theorem \ref{thm linear relations bigolin euler}.
\end{remark}

\section{An explicit computation: \\the small deformations of the Iwasawa manifold}\label{sec_iwasawa}

%\begin{definition}[The Heisenberg group]
We denote by $\mathbb{H}(3;\C)$ the three dimensional Heisenberg group over $\C$
\[
\mathbb{H}(3;\C):=\left\{
\begin{pmatrix}
1 & z^1 & z^3\\
0 & 1 & z^2\\
0& 0& 1
\end{pmatrix}
\in GL(3,\C)\,:\, z^1,z^2,z^3\in\C
\right\}
\]
with the product induced by matrix multiplication, and by $\mathbb{H}(3;\Z[i])$ the subgroup of $\mathbb{H}(3;\C)$ where $z^1,z^2,z^3$ lie in $\Z[i]$. One can check that $\mathbb{H}(3;\C)$ is a connected simply connected complex nilpotent Lie group, with the complex structure induced by $\C^3$. There exists a global frame of left invariant holomorphic $(1,0)$-forms on $\mathbb{H}(3;\C)$ given by
\begin{align*}
\phi^1:=dz^1, && \phi^2:=dz^2, && \phi^3:=dz^3-z^1dz^2,
\end{align*}
with structure equations
\begin{align*}
d\phi^1=0, && d\phi^2=0, && d\phi^3=-\phi^{12}.
\end{align*}
%\end{definition}

\begin{definition}[The Iwasawa manifold]
The Iwasawa manifold is defined as the compact manifold obtained by quotienting $\mathbb{H}(3;\C)$ by the left action of $\mathbb{H}(3;\Z[i])$
\[
\mathbb I_3:=\mathbb{H}(3;\Z[i])\backslash \mathbb{H}(3;\C),
\]
with the complex structure induced by $\mathbb{H}(3;\C)$. The global frame of left invariant holomorphic $(1,0)$-forms on $\mathbb{H}(3;\C)$ descends on the quotient $\mathbb I_3$, and will be denoted with the same symbols.
\end{definition}

%\begin{remark}[The cohomology of the small deformations of the Iwasawa manifold]
We are interested in computing the Bigolin numbers $h^k_{p,q}$ of the small deformations of the Iwasawa manifold. Small deformations and Hodge numbers $h^{p,q}_\delbar$ were computed in \cite{N}, while Aeppli and Bott-Chern numbers $h^{p,q}_A$, $h^{p,q}_{BC}$ were computed in \cite{A}. It turns out that for the small deformations of the complex structure of $\mathbb I_3$, the inclusion
\[
\bigwedge^{p,q}\Lie(\mathbb{H}(3;\C))^*\hookrightarrow \bigwedge^{p,q}\mathbb I_3
\]
induces an isomorphism in Dolbeault, Aeppli and Bott-Chern cohomology, \textit{i.e.}, there exist left invariant (forms which are) representatives for every class in these cohomology spaces on $\mathbb I_3$. By a combination of the results in \cite[Lemma 2.3]{Ste} and either \cite[Proposition 12]{Ste21} or \cite[Theorem C]{Ste23}, it follows that the same result holds for the Bigolin cohomology $H^k_{p,q}$, thus reducing the computation of the invariants $h^k_{p,q}$ to a linear algebra problem.
%\end{remark}

%\begin{remark}[The classes of small deformations]
The description of the small deformations of $\mathbb I_3$ will be as brief as possible, since the topic has been already treated meticulously in other papers like \cite{N,AT,A}. The small deformations are parametrised by a vector of complex numbers $t=(t_{11},t_{12},t_{21},t_{22},t_{31},t_{32})\in B(0,\epsilon)\subset\C$,
obtaining a global frame of $(1,0)$ forms $\phi^1_t$, $\phi^2_t$ and $\phi^3_t$ with structure equations
\[
d\phi^3_t=\sigma_{12}\phi^{12}_t+\sigma_{1\bar1}\phi^{1\bar1}_t+\sigma_{1\bar2}\phi^{1\bar2}_t+\sigma_{2\bar1}\phi^{2\bar1}_t+\sigma_{2\bar2}\phi^{2\bar2}_t,
\]
where each $\sigma_{AB}\in\C$ depends on $t$ and their first order asymptotic behaviour is
\[
\begin{cases}
\sigma_{12}=-1+o(|t|),\\
\sigma_{1\bar1}=t_{21}+o(|t|),\\
\sigma_{1\bar2}=t_{22}+o(|t|),\\
\sigma_{2\bar1}=-t_{11}+o(|t|),\\
\sigma_{2\bar2}=-t_{12}+o(|t|).
\end{cases}
\]
If we set
\begin{align*}
D(t):=\det\begin{pmatrix}
t_{11}&t_{12}\\t_{21}&t_{22}
\end{pmatrix}, &&
S(t):=\begin{pmatrix}
\c{\sigma_{1\bar1}}&\c{\sigma_{2\bar2}}&\c{\sigma_{1\bar2}}&\c{\sigma_{2\bar1}}\\
\sigma_{1\bar1}&\sigma_{2\bar2}&\sigma_{2\bar1}&\sigma_{1\bar2}
\end{pmatrix},
\end{align*}
then we can distinguish the following classes of small deformations:
\begin{enumerate}[(i)]
\item $t_{11}=t_{12}=t_{21}=t_{22}=0$;
\item $(t_{11}=t_{12}=t_{21}=t_{22})\ne(0,0,0,0)$ and $D(t)=0$;
\begin{enumerate}
\item[(ii.a)]$D(t)=0$ and $\rk S=1$;
\item[(ii.b)]$D(t)=0$ and $\rk S=2$;
\end{enumerate}
\item $D(t)\ne0$;
\begin{enumerate}
\item[(iii.a)]$D(t)\ne0$ and $\rk S=1$;
\item[(iii.b)]$D(t)\ne0$ and $\rk S=2$.
\end{enumerate}
\end{enumerate}
We will also need to know that for small deformations in class (i)
\begin{align*}
\sigma_{12}=-1, && \sigma_{1\bar1}=
\sigma_{1\bar2}=
\sigma_{2\bar1}=
\sigma_{2\bar2}=0,
\end{align*}
while for small deformations in class (ii) 
\begin{align*}
 (\sigma_{1\bar1},
\sigma_{1\bar2},
\sigma_{2\bar1},
\sigma_{2\bar2})\ne(0,0,0,0),
\end{align*}
and
\[
\tilde D(t):=\det\begin{pmatrix}
\sigma_{2\bar1}&\sigma_{2\bar2}\\\sigma_{1\bar1}&\sigma_{1\bar2}
\end{pmatrix}=0.
\]
%\end{remark}

%\begin{remark}[The Bigolin numbers]
We now proceed to compute the Bigolin numbers $h^{k}_{p,q}$ on the small deformations of $\mathbb I_3$.
Up to symmetries, by Remark \ref{remark_how_many} there are 7 invariants of the type $h^{k}_{p,q}$ to be determined. 
We draw the parts of the complexes $(\B^\bullet_{p,q},\delta^\bullet_{p,q})$ related to these 7 invariants $h^{k}_{p,q}$.
\begin{align*}
 h^1_{1,1} & & h^1_{1,2},h^2_{1,2} & & h^3_{1,3} \\
\begin{tikzpicture}[scale=0.5]
    \foreach \x in {0,1,...,3}{
    \draw (\x,0)node[below,font=\footnotesize] {\x};
    }
    \foreach \y in {0,1,...,3}{
    \draw (0,\y)node[left,font=\footnotesize] {\y};
    }
    \foreach \x in {0,1,...,3}{
    \foreach \y in {0,1,...,3}{
    \fill[black!30] (\x,\y) circle[radius=1pt];
    }}
    \foreach \x in {0,1,...,1}{
    \foreach \y in {0,1,...,1}{
    \fill[black!80] (\x,\y) circle[radius=2pt];
    }}
\end{tikzpicture}
& &
\begin{tikzpicture}[scale=0.5]
    \foreach \x in {0,1,...,3}{
    \draw (\x,0)node[below,font=\footnotesize] {\x};
    }
    \foreach \y in {0,1,...,3}{
    \draw (0,\y)node[left,font=\footnotesize] {\y};
    }
    \foreach \x in {0,1,...,3}{
    \foreach \y in {0,1,...,3}{
    \fill[black!30] (\x,\y) circle[radius=1pt];
    }}
    \foreach \x in {0,1,...,1}{
    \foreach \y in {0,1,...,2}{
    \fill[black!80] (\x,\y) circle[radius=2pt];
    }}
\end{tikzpicture}
& &
\begin{tikzpicture}[scale=0.5]
    \foreach \x in {0,1,...,3}{
    \draw (\x,0)node[below,font=\footnotesize] {\x};
    }
    \foreach \y in {0,1,...,3}{
    \draw (0,\y)node[left,font=\footnotesize] {\y};
    }
    \foreach \x in {0,1,...,3}{
    \foreach \y in {0,1,...,3}{
    \fill[black!30] (\x,\y) circle[radius=1pt];
    }}
    \foreach \x in {0,1,...,1}{
    \foreach \y in {0,1,...,3}{
    \fill[black!80] (\x,\y) circle[radius=2pt];
    }}
\end{tikzpicture}
\end{align*}

\begin{align*}
 h^3_{2,2} & & h^3_{2,3},h^4_{2,3}\\
\begin{tikzpicture}[scale=0.5]
    \foreach \x in {0,1,...,3}{
    \draw (\x,0)node[below,font=\footnotesize] {\x};
    }
    \foreach \y in {0,1,...,3}{
    \draw (0,\y)node[left,font=\footnotesize] {\y};
    }
    \foreach \x in {0,1,...,3}{
    \foreach \y in {0,1,...,3}{
    \fill[black!30] (\x,\y) circle[radius=1pt];
    }}
    \foreach \x in {0,1,...,2}{
    \foreach \y in {0,1,...,2}{
    \fill[black!80] (\x,\y) circle[radius=2pt];
    }}
\end{tikzpicture}
& &
\begin{tikzpicture}[scale=0.5]
    \foreach \x in {0,1,...,3}{
    \draw (\x,0)node[below,font=\footnotesize] {\x};
    }
    \foreach \y in {0,1,...,3}{
    \draw (0,\y)node[left,font=\footnotesize] {\y};
    }
    \foreach \x in {0,1,...,3}{
    \foreach \y in {0,1,...,3}{
    \fill[black!30] (\x,\y) circle[radius=1pt];
    }}
    \foreach \x in {0,1,...,2}{
    \foreach \y in {0,1,...,3}{
    \fill[black!80] (\x,\y) circle[radius=2pt];
    }}
\end{tikzpicture}
\end{align*}
For the class (i), we can actually compute these invariants just by looking to the double complex of left invariant forms on $\mathbb{H}(3;\C)$ of \cite[Fig. 1]{A} (cf. Proposition \ref{prop zigzag contribute to bigolin})
\begin{align*}
h^1_{1,1}=6, && h^1_{1,2}=5, && h^2_{1,2}=8, && h^3_{1,3}=7,\\
 && h^3_{2,2}=8, && h^3_{2,3}=9, && h^4_{2,3}=6.
\end{align*}
Since the case $t=0$ lies in class (i) and the numbers $h^k_{p,q}$ behaves upper semicontinuously \cite[Corollary D]{Ste}, these values are upper bounds for the invariants $h^k_{p,q}$ of classes (ii) and (iii).

We will compute $h^k_{p,q}$ of classes (ii) and (iii) by computing
\begin{align*}
\delta_{p,q}^k\alpha=0, && \delta_{n-q-1,n-p-1}^{2n-k-1}*\alpha=0
\end{align*}
for all left invariants forms $\alpha\in\B^k_{p,q}$. We choose the diagonal metric given by the fundamental form
\[
\omega=i(\phi^{1\bar1}+\phi^{2\bar2}+\phi^{3\bar3}),
\]
whose associated volume form is
\[
\frac{\omega^3}{3!}=i\phi^{123\bar1\bar2\bar3}.
\]
%\end{remark}

\subsection{$h^1_{1,1}$}\label{sec111}
Having in mind the picture
\begin{center}
\begin{tikzpicture}[scale=0.5]
    \foreach \x in {0,1,...,3}{
    \draw (\x,0)node[below,font=\footnotesize] {\x};
    }
    \foreach \y in {0,1,...,3}{
    \draw (0,\y)node[left,font=\footnotesize] {\y};
    }
    \foreach \x in {0,1,...,3}{
    \foreach \y in {0,1,...,3}{
    \fill[black!30] (\x,\y) circle[radius=1pt];
    }}
    \foreach \x in {0,1,...,1}{
    \foreach \y in {0,1,...,1}{
    \fill[black!80] (\x,\y) circle[radius=2pt];
    }}
\end{tikzpicture}
\end{center}
we set $\alpha=\alpha_1\oplus\alpha_2\in\B^1_{1,1}=A^{0,1}\oplus A^{1,0}$ and recall that $\alpha\in\H^{1}_{1,1}$ if and only if
\begin{align}\label{eq_harmonic_111}
\del\alpha_1+\delbar\alpha_2=0, && \delbar^*\alpha_1+\del^*\alpha_2=0.
\end{align}
We can write
\[
\alpha_1=A_{\bar1}\phi^{\bar1}_t+A_{\bar2}\phi^{\bar2}_t+A_{\bar3}\phi^{\bar3}_t,
\]
\[
\alpha_2=A_{1}\phi^{1}_t+A_{2}\phi^{2}_t+A_{3}\phi^{3}_t,
\]
\[
*\alpha_1=iA_{\bar1}\phi^{23\bar1\bar2\bar3}_t-iA_{\bar2}\phi^{13\bar1\bar2\bar3}_t+iA_{\bar3}\phi^{12\bar1\bar2\bar3}_t,
\]
\[
*\alpha_2=-iA_{1}\phi^{123\bar2\bar3}_t+iA_{2}\phi^{123\bar1\bar3}_t-iA_{3}\phi^{123\bar1\bar2}_t,
\]
and compute
\begin{align*}
\del\alpha_1+\delbar\alpha_2&=A_{\bar3}(-\c{\sigma_{1\bar1}}\phi^{1\bar1}-\c{\sigma_{2\bar1}}\phi^{1\bar2}-\c{\sigma_{1\bar2}}\phi^{2\bar1}-\c{\sigma_{2\bar2}}\phi^{2\bar2})\\
&+A_{3}(\sigma_{1\bar1}\phi^{1\bar1}+\sigma_{1\bar2}\phi^{1\bar2}+\sigma_{2\bar1}\phi^{2\bar1}+\sigma_{2\bar2}\phi^{2\bar2}),
\end{align*}
\[
\del*\alpha_1+\delbar*\alpha_2=0.
\]
Therefore $h^1_{1,1}=6-\rk S(t)$, \textit{i.e.}, $h^1_{1,1}=5$ in classes (ii.a) and (iii.a), while $h^1_{1,1}=4$ in classes (ii.b) and (iii.b).

\subsection{$h^1_{1,2}$}\label{sec112}
Having in mind the picture
\begin{center}
\begin{tikzpicture}[scale=0.5]
    \foreach \x in {0,1,...,3}{
    \draw (\x,0)node[below,font=\footnotesize] {\x};
    }
    \foreach \y in {0,1,...,3}{
    \draw (0,\y)node[left,font=\footnotesize] {\y};
    }
    \foreach \x in {0,1,...,3}{
    \foreach \y in {0,1,...,3}{
    \fill[black!30] (\x,\y) circle[radius=1pt];
    }}
    \foreach \x in {0,1,...,1}{
    \foreach \y in {0,1,...,2}{
    \fill[black!80] (\x,\y) circle[radius=2pt];
    }}
\end{tikzpicture}
\end{center}
we let $\alpha=\alpha_1\oplus\alpha_2\in\B^1_{1,2}=A^{0,1}\oplus A^{1,0}$ as in Subsection \ref{sec111}.
Then $\alpha\in\H^{1}_{1,2}$ if and only if it satisfies \eqref{eq_harmonic_111} plus the extra condition $\delbar\alpha_1=0$. Since
\[
\delbar\alpha_1=A_{\bar3}\c{\sigma_{12}}\phi^{\bar1\bar2},
\]
it follows $h^1_{1,2}=4$ for both classes (ii) and (iii).

\subsection{$h^2_{1,2}$}\label{sec212}
Having in mind the same picture of Subsection \ref{sec112}, we set $\alpha=\alpha_1\oplus\alpha_2\in\B^{2}_{1,2}=A^{0,2}\oplus A^{1,1}$ and recall that $\alpha\in\H^{2}_{1,2}$ if and only if
\begin{align}\label{eq_harmonic_212}
\del\alpha_1+\delbar\alpha_2=0, && \delbar^*\alpha_1+\del^*\alpha_2=0, &&
\delbar^*\alpha_2=0.
\end{align}
We can write
\[
\alpha_1=A_{\bar1\bar2}\phi^{\bar1\bar2}_t+A_{\bar1\bar3}\phi^{\bar1\bar3}_t+A_{\bar2\bar3}\phi^{\bar2\bar3}_t,
\]
\[
\alpha_2=\sum_{i,k=1}^3A_{j\bar k}\phi^{j\bar k}_t,
\]
\[
*\alpha_1=iA_{\bar1\bar2}\phi^{3\bar1\bar2\bar3}_t-iA_{\bar1\bar3}\phi^{2\bar1\bar2\bar3}_t+iA_{\bar2\bar3}\phi^{1\bar1\bar2\bar3}_t,
\]
\begin{align*}
*\alpha_2&=-iA_{1\bar1}\phi^{23\bar2\bar3}_t-iA_{2\bar2}\phi^{13\bar1\bar3}_t-iA_{3\bar3}\phi^{12\bar1\bar2}_t\\
&+iA_{1\bar2}\phi^{13\bar2\bar3}_t-iA_{1\bar3}\phi^{12\bar2\bar3}_t+iA_{2\bar3}\phi^{12\bar1\bar3}_t\\
&+iA_{2\bar1}\phi^{23\bar1\bar3}_t-iA_{3\bar1}\phi^{23\bar1\bar2}_t+iA_{3\bar2}\phi^{13\bar1\bar2}_t
\end{align*}
and compute
\begin{align*}
\del\alpha_1+\delbar\alpha_2&=A_{\bar1\bar3}(-\c{\sigma_{2\bar1}}\phi^{1\bar1\bar2}_t-\c{\sigma_{2\bar2}}\phi^{2\bar1\bar2}_t)+A_{\bar2\bar3}(\c{\sigma_{1\bar1}}\phi^{1\bar1\bar2}_t+\c{\sigma_{1\bar2}}\phi^{2\bar1\bar2}_t)\\
&+A_{3\bar1}(-\sigma_{1\bar2}\phi^{1\bar1\bar2}_t-\sigma_{2\bar2}\phi^{21\bar2}_t)+A_{3\bar2}(\sigma_{1\bar1}\phi^{1\bar1\bar2}_t+\sigma_{2\bar1}\phi^{2\bar1\bar2}_t)\\
&+A_{1\bar3}(-\c{\sigma_{12}}\phi^{1\bar1\bar2}_t)+A_{2\bar3}(-\c{\sigma_{12}}\phi^{2\bar1\bar2}_t)+A_{3\bar3}(\delbar\phi^3_t\wedge\phi^{\bar3}_t-\phi^3_t\wedge\delbar\phi^{\bar3}_t),
\end{align*}
\begin{align*}
\del*\alpha_1&+\delbar*\alpha_2=\\
&=\left(iA_{\bar1\bar2}(\sigma_{12})+iA_{1\bar1}(-\sigma_{1\bar1})+iA_{2\bar2}(-\sigma_{2\bar2})+iA_{1\bar2}(-\sigma_{2\bar1})+iA_{2\bar1}(-\sigma_{1\bar2})\right)\phi^{12\bar1\bar2\bar3}_t,
\end{align*}
\[
\del*\alpha_2=\left(iA_{1\bar1}(-\c{\sigma_{1\bar1}})+iA_{2\bar2}(-\c{\sigma_{2\bar2}})+iA_{1\bar2}(-\c{\sigma_{1\bar2}})+iA_{2\bar1}(-\c{\sigma_{2\bar1}})\right)\phi^{123\bar1\bar2}_t.
\]
Therefore $h^{2}_{1,2}=12-5=7$ for both classes (ii) and (iii).

\subsection{$h^3_{1,3}$}\label{sec313}
Having in mind the picture
\begin{center}
\begin{tikzpicture}[scale=0.5]
    \foreach \x in {0,1,...,3}{
    \draw (\x,0)node[below,font=\footnotesize] {\x};
    }
    \foreach \y in {0,1,...,3}{
    \draw (0,\y)node[left,font=\footnotesize] {\y};
    }
    \foreach \x in {0,1,...,3}{
    \foreach \y in {0,1,...,3}{
    \fill[black!30] (\x,\y) circle[radius=1pt];
    }}
    \foreach \x in {0,1,...,1}{
    \foreach \y in {0,1,...,3}{
    \fill[black!80] (\x,\y) circle[radius=2pt];
    }}
\end{tikzpicture}
\end{center}
 we set $\alpha=\alpha_1\oplus\alpha_2\in\B^3_{1,3}=A^{0,3}\oplus A^{1,2}$. Then $\alpha\in\H^3_{1,3}$ if and only if
\begin{align*}
\del\alpha_1+\delbar\alpha_2=0, && \delbar^*\alpha_1+\del^*\alpha_2=0, && \delbar^*\alpha_2=0.
\end{align*}
We can write
\[
\alpha_1=A_{\bar1\bar2\bar3}\phi^{\bar1\bar2\bar3}_t,
\]
\[
\alpha_2=\sum_{\substack{1\le j,k,l\le 3\\k<l}}A_{j\bar k\bar l}\phi^{j\bar k\bar l}_t,
\]
\[
*\alpha_1=iA_{\bar1\bar2\bar3}\phi^{\bar1\bar2\bar3}_t
\]
\begin{align*}
*\alpha_2&=-iA_{1\bar1\bar2}\phi^{3\bar2\bar3}_t+iA_{1\bar1\bar3}\phi^{2\bar2\bar3}_t-iA_{1\bar2\bar3}\phi^{1\bar2\bar3}_t\\
&+iA_{2\bar1\bar2}\phi^{3\bar1\bar3}_t-iA_{2\bar1\bar3}\phi^{2\bar1\bar3}_t+iA_{2\bar2\bar3}\phi^{1\bar1\bar3}_t\\
&-iA_{3\bar1\bar2}\phi^{3\bar1\bar2}_t+iA_{3\bar1\bar2}\phi^{3\bar1\bar2}_t-iA_{3\bar2\bar3}\phi^{1\bar1\bar2}_t
\end{align*}
and compute
\[
\del\alpha_1+\delbar\alpha_2=A_{3\bar1\bar3}(-\sigma_{1\bar2}\phi^{1\bar1\bar2\bar3}_t-\sigma_{2\bar2}\phi^{2\bar1\bar2\bar3}_t)+A_{3\bar2\bar3}(\sigma_{1\bar1}\phi^{1\bar1\bar2\bar3}_t+\sigma_{2\bar1}\phi^{2\bar1\bar2\bar3}_t)
\]
\[
\del*\alpha_1+\delbar*\alpha_2=iA_{1\bar1\bar2}(-\sigma_{1\bar1}\phi^{1\bar1\bar2\bar3}_t-\sigma_{2\bar1}\phi^{2\bar1\bar2\bar3}_t)+iA_{2\bar1\bar2}(-\sigma_{1\bar2}\phi^{1\bar1\bar2\bar3}_t-\sigma_{2\bar2}\phi^{2\bar1\bar2\bar3}_t),
\]
\begin{align*}
\del*\alpha_2&=(iA_{3\bar1\bar2}(-\sigma_{12})+iA_{1\bar1\bar3}(\c{\sigma_{1\bar1}})+iA_{1\bar2\bar3}(\c{\sigma_{1\bar2}})+iA_{2\bar1\bar3}(\c{\sigma_{2\bar1}})+iA_{2\bar2\bar3}(\c{\sigma_{2\bar2}}))\phi^{12\bar1\bar2}_t\\
&+A_{1\bar1\bar2}(-\sigma_{12}\phi^{12\bar2\bar3}_t-\phi^{3\bar2}_t\wedge\del\phi^{\bar3}_t)+A_{2\bar1\bar2}(\sigma_{12}\phi^{12\bar1\bar3}_t+\phi^{3\bar1}_t\wedge\del\phi^{\bar3}_t).
\end{align*}
Therefore $h^3_{1,3}=10-3-\rk D(t)$, \textit{i.e.}, $h^3_{1,3}=6$ in class (ii), while $h^3_{1,3}=5$ in class (iii).

\subsection{Remaining cases}
With similar but heavier computations it is possible to show
\begin{align*}
 h^3_{2,2}=8, && h^3_{2,3}=9, && h^4_{2,3}=6
\end{align*}
for both classes (ii) and (iii).

\subsection{Conclusions}
We report a table of the invariants that we have just computed for the small deformations of $\mathbb{I}_3$. We invite the reader to compare it with the table in the appendix of \cite{A}.
\begin{center}
\begin{tabular}{ccccccccc}
\toprule
$H^k_{p,q}$ & $h^{1}_{1,1}$ & $h^{1}_{1,2}$ & $h^{2}_{1,2}$  & $h^{3}_{1,3}$  & $h^{3}_{2,2}$  & $h^{3}_{2,3}$ & $h^{4}_{2,3}$ \\
\midrule

(i) & 6 & 5 & 8  & 7  & 8  & 9 & 6 \\

(ii.a) & 5 & 4 & 7  & 6  & 8  & 9 & 6 \\

(ii.b) & 4 & 4 & 7  & 6  & 8  & 9 & 6 \\

(iii.a) & 5 & 4 & 7  & 5  & 8  & 9 & 6 \\

(iii.b) & 4 & 4 & 7  & 5  & 8  & 9 & 6 \\
\bottomrule
\end{tabular}
\end{center}

Summing up, we have shown that the invariants $h^{k}_{p,q}$ can distinguish from different classes of the small deformations of the Iwasawa manifold $\mathbb I_3$. In particular, they can distinguish between all classes (i), (ii.a), (ii.b), (iii.a) and (iii.b).

This should be compared with the results in \cite{N}, where it is shown that the invariants $h^{p,q}_\delbar$ can distinguish between classes (i), (ii) and (iii), and with the results in \cite{A}, where it is proved that the invariants $h^{p,q}_A$ and $h^{p,q}_{BC}$ can distinguish between all classes (i), (ii.a), (ii.b), (iii.a) and (iii.b). In this sense, in this example, we can say that the invariants $h^{k}_{p,q}$ are  as powerful as $h^{p,q}_A$ and $h^{p,q}_{BC}$, in order to distinguish classes of complex structures.

As a final remark, we would like to point out that in \cite{A}, the (only) invariant providing the distinction between subclasses (a) and (b) is $h^{1,1}_A$($=h^2_{1,1}$) (and dually $h^{2,2}_{BC}$). Here, we are able to distinguish between subclasses (a) and (b) via the invariant $h^1_{1,1}$. Thus, the complex providing the whole distinction between subclasses (a) and (b) is $(\B^\bullet_{1,1},\delta^\bullet_{1,1})$.

\begin{remark}[Application of Theorem \ref{thm zigzags 3manifolds}]\label{rmk flavi}
In \cite{F} the double complexes of some fixed choices of small deformations in classes (ii.a), (iii.a) and (iii.b) of the Iwasawa manifold are computed. As a consequence of the above computations and of Theorem \ref{thm zigzags 3manifolds}, the double complexes of these fixed choices are actually the double complexes of all the small deformations in classes (ii.a), (iii.a) and (iii.b) of the Iwasawa manifold. Furthermore, since from the above table and \cite[Appendix]{A} we know all Betti, Dolbeault and Aeppli numbers plus $h^1_{1,1}$, $h^1_{1,2}$ and $h^4_{2,3}$, we can use Table \ref{table inverse}, as explained in Remark \ref{remark inverse}, to deduce the multiplicities of all the zigzags also in the remaining class (ii.b); they are reported in Table \ref{table class ii.b}. This is coherent with \cite[Section 9.1]{Ste22}, where the same result is obtained by a careful computation of the multiplicities of all the zigzags contributing to Betti, Dolbeault and Bott-Chern numbers.
\begin{table}[!ht]
\captionof{table}{Multiplicities of zigzags of class (ii.b) of small deformations of the Iwasawa manifold}\label{table class ii.b}
    \centering
    
    \begin{tabular}{|l|l|l|l|l|l|l|l|l|l|l|l|l|l|l|l|l|}
    \hline
        A & B & C & D & E & F & G & H & I & L & M & N & O & P & Q & R & S \\ \hline
        2 & 0 & 1 & 0 & 0 & 0 & 0 & 1 & 0 & 1 & 0 & 0 & 2 & 1 & 0 & 1 & 4 \\ \hline
    \end{tabular}
\end{table}
\end{remark}

\section{Almost complex manifolds}\label{sec_almost}
In this section we will consider an almost complex manifold $(M,J)$ of real dimension $2n$. In general the exterior derivative on $(p,q)$-forms acts as
\begin{equation*}
d:A^{p,q}\to A^{p+2,q-1}\oplus A^{p+1,q}\oplus A^{p,q+1}\oplus A^{p-1,q+2},
\end{equation*}
where we denote the four components of $d$ by
\begin{equation*}
d=\mu+\del+\delbar+\c\mu.
\end{equation*}
From the relation $d^2=0$, we derive
\begin{equation*}
\begin{cases}
\mu^2=0,\\
\mu\del+\del\mu=0,\\
\del^2+\mu\delbar+\delbar\mu=0,\\
\del\delbar+\delbar\del+\mu\c\mu+\c\mu\mu=0,\\
\delbar^2+\c\mu\del+\del\c\mu=0,\\
\c\mu\delbar+\delbar\c\mu=0,\\
\c\mu^2=0.
\end{cases}
\end{equation*}

\begin{remark}[Lack of a complete generalisation]
If the almost complex structure $J$ is not integrable, there are at least two problems which seem insurmountable in order to define a complete generalisation of the elliptic complex $(\B^\bullet_{p,q},\delta_{p,q}^\bullet)$ of section \ref{section_complex}:
\begin{enumerate}
\item as in the integrable case, we could define the second half of the complex by applying repeatedly the exterior derivative $d$ starting from $A^{p+1,q+1}$: in general this complex is not elliptic if $J$ is not integrable (one can verify that also the first half of the complex, defined by duality, is not elliptic);
\item there is no valid replacement for the operator $\del\delbar$ for the middle differential of the complex.
\end{enumerate}
\end{remark}

\begin{definition}[A partial generalisation]\label{definition_complex_almost}
Since $\del\delbar+\delbar\del=0$ when applied to $A^{0,0}$ or to $A^{n-1,n-1}$, the following sequences of maps are complexes
\[
A^{0,0}\overset{d}{\longrightarrow}A^{0,1}\oplus A^{1,0}\overset{\del\oplus\delbar}{\longrightarrow} A^{1,1},
\]
\[
A^{n-1,n-1}\overset{d}{\longrightarrow}A^{n-1,n}\oplus A^{n,n-1}\overset{d}{\longrightarrow} A^{n,n},
\]
which can be verified to be elliptic. In fact, they coincide respectively with the first two maps of the complex $(\B^\bullet_{1,1},\delta_{1,1}^\bullet)$ and with the last two maps of the complex $(\B^\bullet_{n-2,n-2},\delta_{n-2,n-2}^\bullet)$.
\end{definition}

\begin{definition}[Cohomology]
We can define the following two spaces
\[
H^1_{\B}:=\frac{\ker \del\oplus\delbar\cap A^{0,1}\oplus A^{1,0}}{d A^{0,0}},
\]
\[
H^{2n-1}_{\B}:=\frac{\ker d\cap A^{2n-1}_\C}{d A^{n-1,n-1}},
\]
which are the cohomology spaces of the two complexes just defined. Note that by definition we have the inclusion
\[
H^1_{dR}\subseteq H^1_\B.
\]
\end{definition}

\begin{definition}[Elliptic operators]
We endow our almost complex manifold $(M,J)$ with an almost Hermitian metric $g$. We can write the formal adjoints of the operators in the complexes in Definition \ref{definition_complex_almost} as
\[
A^{0,0}\overset{d^*}{\longleftarrow}A^{0,1}\oplus A^{1,0}\overset{d^*}{\longleftarrow} A^{1,1},
\]
\[
A^{n-1,n-1}\overset{\delbar^*\oplus\del^*}{\longleftarrow}A^{n-1,n}\oplus A^{n,n-1}\overset{d^*}{\longleftarrow} A^{n,n}.
\]
The elliptic and formally self adjoint operators
\[
\square^1_\B:A^{0,1}\oplus A^{1,0}\to A^{0,1}\oplus A^{1,0},
\]
\[
\square^{2n-1}_\B:A^{n-1,n}\oplus A^{n,n-1}\to A^{n-1,n}\oplus A^{n,n-1}
\]
associated to the complexes are then defined by
\[
\square^1_\B:=dd^*+d^*(\del\oplus\delbar),
\]
\[
\square^{2n-1}_\B:=d(\delbar^*\oplus\del^*)+d^*d.
\]
\end{definition}

\begin{remark}[Hodge theory]
Assume that the almost Hermitian manifold $(M,J,g)$ is compact. It follows that the spaces of harmonic forms
\[
\H^1_\B:=\ker\square^1_\B=\ker d^*\cap\ker\del\oplus\delbar,
\]
\[
\H^{2n-1}_\B:=\ker\square^{2n-1}_\B=\ker d\cap \ker \delbar^*\oplus\del^*
\]
have finite dimension $h^1_\B$ and $h^{2n-1}_\B$. Note that since
\[
\ker\Delta_d=\ker d\cap \ker d^*,
\]
then there are inclusions
\begin{align*}
\ker\Delta_d\cap A^k_\C\subseteq\H^k_\B
\end{align*}
for $k=1,2n-1$.
 Moreover, there are $L^2$ orthogonal decompositions
\[
A^1_\C=\H^1_\B\oplus d A^{0,0}\oplus d^* A^{1,1},
\]
\[
A^{2n-1}_\C=\H^{2n-1}_\B\oplus d A^{n-1,n-1}\oplus d^* A^{n,n}.
\]
It follows that
\[
\ker\del\oplus\delbar\cap A^{0,1}\oplus A^{1,0}=\H^1_\B\oplus d A^{0,0},
\]
\[
\ker d \cap A^{2n-1}_\C=\H^{2n-1}_\B\oplus d A^{n-1,n-1},
\]
therefore there are isomorphisms
\begin{align*}
H^1_\B\simeq \H^1_\B, & & H^{2n-1}_\B\simeq\H^{2n-1}_\B.
\end{align*}
In particular the numbers $h^1_\B$ and $h^{2n-1}_\B$ are metric independent.
The Hodge $*$ operator defines an isomorphism
\[
\H^1_\B\simeq\H^{2n-1}_\B,
\]
therefore $h^{1}_\B=h^{2n-1}_\B$.
\end{remark}

\begin{remark}[Upper semi-continuity]
By \cite[Theorem 4.3, Chapter 4]{KM}, on a given compact manifold, the dimension of the kernel of a smooth family of elliptic operators is upper semi-continuous. Therefore, it follows that $h^1_\B$ is upper semi-continuous along smooth deformations of the almost complex structure. The same result for $h^k_{p,q}$ on compact complex manifolds is shown in \cite[Corollary D]{Ste}.
\end{remark}

\begin{problem}[Explicit computation of $h^1_\B$]
The first immediate problem is how to compute this new invariant $h^1_\B$ in explicit examples. When the almost complex structure $J$ is integrable, as we did in Section \ref{sec_iwasawa}, there are many classes of nilmanifolds where the inclusion of left invariant forms into the de Rham algebra induces an isomorphism in the cohomology $H^k_{p,q}$. 

It seems then reasonable to ask if there are as well classes of nilmanifolds carrying left invariant almost complex structures for which the inclusion of left invariant forms into the de Rham algebra induces an isomorphism in the cohomology $H^1_\B$.
\end{problem}

\end{document}